\documentclass[12pt, reqno]{amsart}

\usepackage{graphics, stackrel}
\usepackage{amsmath, amssymb, amsthm}
\usepackage{graphicx}
\usepackage{verbatim}
\usepackage{amsfonts}
\usepackage{natbib}
\usepackage{enumitem}
\usepackage{mathtools}

\usepackage[T1]{fontenc}  
\usepackage{lmodern}  
\usepackage{textcomp}
\usepackage{breakurl}    
\usepackage{caption}
\usepackage{subcaption}

\usepackage{fancyvrb}
\usepackage[dvipsnames]{xcolor}
\usepackage{mdwlist}

\usepackage[citecolor=blue, 
            colorlinks=true, 
            linkcolor=blue, 
            urlcolor=black]{hyperref}

\usepackage{enumitem}
\setlist[enumerate]{itemsep=2pt,topsep=3pt}
\setlist[itemize]{itemsep=2pt,topsep=3pt}
\setlist[enumerate,1]{label={\upshape (\alph*)}}

\usepackage{mathrsfs}  % caligraphic
\usepackage{bbm}
\usepackage{bm}        % bold symbols

\usepackage[left=1.25in, right=1.25in, top=1.0in, 
bottom=1.15in, includehead, includefoot]{geometry}

\renewcommand{\leq}{\leqslant}
\renewcommand{\geq}{\geqslant}

\providecommand{\inner}[1]{\left\langle{#1}\right\rangle}

\usepackage[ruled, linesnumbered]{algorithm2e}
\newcommand{\algoref}[1]{Algorithm~\hyperlink{#1}{\ref*{#1}}}
\newcommand{\setntn}[2]{ \{ #1 : #2 \} }

\newcommand{\1}{\mathbbm 1}

\newcommand*\diff{\mathop{}\!\mathrm{d}}

\newcommand{\bB}{\mathscr B}

\newcommand{\dD}{\mathcal D}

\newcommand{\RR}{\mathbbm R}

\newcommand{\NN}{\mathbbm N}

\newcommand{\TT}{\mathbbm T}

\newcommand{\Asf}{\mathsf A}

\newcommand{\Gsf}{\mathsf G}
\newcommand{\Xsf}{\mathsf X}

\newcommand{\Wsf}{\mathsf W}

\renewcommand{\phi}{\varphi}
\renewcommand{\epsilon}{\varepsilon}

\makeatletter
\def\namedlabel#1#2{\begingroup
    #2%
    \def\@currentlabel{#2}%
    \phantomsection\label{#1}\endgroup
}
\makeatother

\theoremstyle{plain}
\newtheorem{theorem}{Theorem}[section]

\newtheorem{lemma}[theorem]{Lemma}
\newtheorem{proposition}[theorem]{Proposition}

\theoremstyle{definition}

\newcommand{\navy}[1]{\textit{\textcolor{Blue}{#1}}}

\usepackage[normalem]{ulem}
\makeatletter
\renewcommand{\@evenhead}{\small\rlap{\thepage}\hfill \scriptsize 
JOHN STACHURSKI, JINGNI YANG, ZIYUE YANG\hfill}
\makeatother

\begin{document}
\title[Global Optimality]
{Dynamic Programming: \\ From Local Optimality to Global Optimality}

\author{
    John Stachurski\textsuperscript{$\ast$}, 
    Jingni Yang\textsuperscript{$\dagger$}, 
    Ziyue Yang\textsuperscript{$\ddagger$}
}

\thanks{\textsuperscript{$\ast$}Research School of Economics, 
Australian National University. 
\href{mailto:john.stachurski@anu.edu.au}{\texttt{john.stachurski@anu.edu.au}}}
\thanks{\textsuperscript{$\dagger$}School of Economics, University of Sydney. 
\href{mailto:jingni.yang@sydney.edu.au}{\texttt{jingni.yang@sydney.edu.au}}}
\thanks{\textsuperscript{$\ddagger$}Research School of Economics, 
Australian National University. 
\href{mailto:humphrey.yang@anu.edu.au}{\texttt{humphrey.yang@anu.edu.au}}}

\begin{abstract}
    In the theory of dynamic programming, an optimal policy is a policy whose
    lifetime value dominates that of all other policies from every possible
    initial condition in the state space. This raises a natural question: when
    does optimality from a single state imply optimality from every state? 
    Working in a general setting, we 
    provide sufficient conditions for this property that relate to reachability and
    irreducibility. Our results have significant implications for 
    modern policy-based algorithms used to solve large-scale dynamic programs.
    We illustrate our findings by applying them to an optimal savings problem
    via an algorithm that implements gradient ascent in a policy space
    constructed from neural networks.
\end{abstract}

\maketitle
\section{Introduction}

Dynamic programming (DP) is a major branch of optimization theory with a vast
range of applications. Within economics alone, applications extend from monetary
and fiscal policy to asset pricing, unemployment, firm investment, wealth
dynamics, commodity pricing, sovereign default, the division of labor, natural
resource extraction, human capital accumulation, retirement decisions, portfolio
choice, and dynamic pricing.  Dynamic programs that include uncertainty are
often called Markov decision processes (MDPs) and the theory of such processes
has been extensively developed (see, e.g., \cite{bauerle2011markov},
\cite{hernandez2012discrete}, \cite{bertsekas2012dynamic}, or
\cite{bertsekas2022abstract}).  

Much of the recent surge in interest in MDPs has been fueled by artificial
intelligence and reinforcement learning (see, e.g., \cite{bertsekas2021rollout}
or \cite{kochenderfer2022algorithms}). In recent years, researchers solving
large-scale MDPs have moved away from value-based methods and towards
policy-based methods, such as policy gradient ascent (see, e.g., \cite{Sutton1999PolicyGM}, 
\cite{Lan2023ImprovedCE}, \cite{Kumar2023PolicyGF}, \cite{murphy2024reinforcementlearningoverview}).  
These policy-based methods seek to maximize a real-valued objective, such as
\begin{equation}\label{eq:msig}
    m(\sigma) \coloneq \int v_\sigma(x) \rho(\diff x)
    \qquad (\sigma \in \Sigma).
\end{equation}

Here $\sigma$ is a policy for a given MDP, mapping some state space $\Xsf$ into
a specified action space, $v_\sigma(x)$ represents the lifetime value of
following the fixed policy $\sigma$, conditional on initial state $x$, and
$\rho$ is a given ``initial distribution.''  In practice, each $\sigma$ is
typically represented by a neural network. Policy-based methods often handle
large problems with continuous state and action spaces more efficiently than
traditional value-based methods such as value function iteration (VFI). The
framework has led to numerous successful algorithms for solving complex
decision-making problems, including Trust Region Policy
Optimization~\citep{Schulman2015TRPO}, Asynchronous Advantage
Actor-Critic~\citep{mnih2016asynchronous}, and Proximal Policy
Optimization~\citep{schulman2017proximal}. 

Despite these successes, there is one significant disadvantage of policy
gradient methods: unlike traditional DP algorithms, these methods are not
guaranteed to converge to an optimal policy, even in relatively straightforward
settings. To understand the issues at hand, recall
that an \emph{optimal policy} is a $\sigma \in \Sigma$ such that
$v_\sigma(x) = \max_{s \in \Sigma} v_s(x)$ for every $x \in \Xsf$; that is, a
policy $\sigma$ such that following this policy in every period leads to maximum
lifetime value from every initial state $x$. Even if one attains a global
maximum in \eqref{eq:msig}, with maximization over all $\sigma \in \Sigma$,
there is no guarantee that the resulting policy will be an optimal policy. 
Indeed, the result depends in general on the initial distribution $\rho$, which is not
one of the primitives of the DP problem.

Regarding this issue, the prevailing view is that relative optimality depends on
the initial distribution $\rho$ having large support.  For example, in a major
study of policy gradient methods, \cite{bhandari2024global} state that ``Policy
gradient methods have poor convergence properties if applied without an
exploratory initial distribution'' \citep[p.~1910]{bhandari2024global}. 

One disadvantage of choosing $\rho$ with large support is that computing the
integral \eqref{eq:msig} becomes computationally expensive, since the function
$v_\sigma(x)$ is expensive to evaluate.  Another is that, in general, this
integral cannot be computed exactly, and, moreover, the implications of a
relatively poor approximation to the integral for optimality theory are not
well-known.  A third issue is that existing conditions for ``sufficiently large
supports'' are difficult to test in practice, making it hard to understand
whether or not the relevant theory applies.  (We return to this last point
below.)

These observations motivate the following question: Does there exist a large and
important class of models for which it is permissible to choose $\rho$ with
\emph{small} support, while still maintaining the property that maximization of this
criterion over all $\sigma \in \Sigma$ yields a (globally) optimal policy?
At the extreme, when can $\rho$ be concentrated at a single point $x$, so that
the maximization criterion $m(\sigma)$ from \eqref{eq:msig} is just
$v_\sigma(x)$? In other words, when does optimality at a single state imply
optimality at every state? 

We show that, for standard MDPs on general state spaces, strong irreducibility of the
Markov dynamics generated by $\sigma$ is sufficient for this property.
Specifically, if a policy is optimal at a single state and has a strongly irreducible
transition kernel, then this optimality propagates throughout the entire state
space, making the policy globally optimal.  (Here ``strong irreducibility'' refers to
irreducibility of positive operators under the standard Riesz space definition.)
In addition, we show that weaker forms of irreducibility are also sufficient for this
result when $\sigma$ is also continuous.  

Since gradient policy methods are commonly applied to problems with continuous
state and action spaces, we avoid discreteness restrictions.  In particular, for
the MDPs that we consider, the state and action spaces can be arbitrary metric
spaces.  We also avoid placing restrictions on the class of MDPs under
consideration, adopting only mild regularity conditions that imply existence of solutions.
Focusing on problems where solutions exist is natural for this line of research, since
we wish to examine conditions under which local optima imply global optima.

Papers examining theoretical properties of policy gradient methods have some
connection to our work. For example, \cite{bhandari2024global} examine when
policy gradient ascent (actually decent) yields a globally optimal policy. To
step from local to global optimality, they restrict the classes of MDPs under
consideration and require that the initial distribution $\rho$ dominates the
discounted state occupancy measure under the optimal policy (i.e, the measure is
absolutely continuous with respect to $\rho$).   This condition is difficult to
verify in practice, since the optimal policy is \emph{a priori} unknown (and, in
fact, the object of the whole computational procedure).  Here we avoid both
large support restrictions on $\rho$ and explicit restrictions on the optimal
policy. (At the same time, \cite{bhandari2024global} supply many valuable new
results concerning the policy gradient algorithm, and we make no direct
contribution to this analysis.)

Other papers that examine theoretical properties of gradient policy methods
include \cite{khodadadian2021linear}, \cite{agarwal2021theory}, and
\cite{xiao2022convergence}. However, in these papers, the focus is on proving
the convergence of $\int v_\sigma \diff \rho$ to its maximal value, rather than
obtaining global convergence from local convergence (as we do here).  At the
same time, these papers provide valuable rates of convergence for specific
algorithms, which we do not discuss. A related line of research focuses on
average-optimal policies in finite-state MDPs by leveraging specific state space
structures under some policies. This includes exploring unichain, multichain,
communicating, and weakly communicating MDPs to study algorithmic convergence
\citep{Bartlett2009REGAL, puterman2014markov}. Our results focus on the
continuous state setting, where gradient policy methods are usually applied.

In an application connected to the income fluctuation problem, we examine policy
gradient methods under different irreducibility settings using neural network
representations of policy functions.  In this we follow a recent trend in
economics towards solving large scale dynamic problems using neural networks to
represent equilibrium objects (see, e.g., \cite{maliar2021deep} or
\cite{friedl2023deep}).  However, these papers do not examine the link between
local and global optimality. In fact they do not use the class of policy gradient methods 
described above, focusing instead on minimizing Euler or Bellman residuals.

The paper is structured as follows.  Section~\ref{s:mdps} provides background on
MDPs.  Section~\ref{s:lg} states our main results in a general setting.
Section~\ref{s:tcs} examines how the irreducibility condition from
Section~\ref{s:lg} can be weakened while still obtaining some transmission of
optimality across states. Section~\ref{s:os} illustrates our theoretical results
in the context of a benchmark optimal savings problem. Section~\ref{s:ext}
outlines avenues for future work.

\section{Markov Decision Processes}\label{s:mdps}

In this section, we review essential properties of Markov decision processes and
state a technical lemma that will be applied in our main results.

\subsection{Preliminaries}\label{ss:prel}

Let $\Xsf$ and $\Asf$ be metric spaces, let $\bB$ be the Borel subsets of
$\Xsf$, let $b\Xsf$  be the set of bounded Borel measurable functions from
$\Xsf$ to $\RR$, and let $cb\Xsf$ be the continuous functions in $b\Xsf$.  Both
$b\Xsf$ and $cb\Xsf$ are paired with the supremum norm $\| \cdot \|$ and the
pointwise partial order $\leq$.  For example, $f \leq g$ indicates that $f(x)
\leq g(x)$ for all $x \in \Xsf$. A map $M$ from $b\Xsf$ to itself is called
\navy{order preserving} if $f \leq g$ implies $Mf \leq Mg$.  
Absolute values are applied pointwise, so that
$|f|$ is the function $x \mapsto |f(x)|$.

A \navy{transition kernel} on $\Xsf$ is a function $Q$ from $\Xsf \times \bB$ to
$[0,1]$ such that $x \mapsto Q(x,B)$ is Borel measurable for all $B \in \bB$ and
$B \mapsto Q(x,B)$ is a Borel probability measure for all $x \in \Xsf$.  To any
such kernel $Q$ we associate a bounded linear operator on $b\Xsf$,
often referred to as its \navy{Markov operator} and also denoted by $Q$, via
\begin{equation}\label{eq:mo}
    f \mapsto Qf, \qquad (Qf)(x) = \int f(x') Q(x, \diff x').
\end{equation}
Typically, $(Qf)(x)$ represents the expectation of $f(X_{t+1})$
given that $X_t = x$ and $X_{t+1}$ is drawn from $Q(x, \diff x')$.

As usual, the positive cone of $b\Xsf$, denoted here by $b\Xsf_+$, is all
$v\in b\Xsf$ with $v \geq 0$.  Let $b\Xsf'$ be the dual space of
$b\Xsf$ and let $b\Xsf'_+$ to be the positive cone of $b\Xsf'$. 
The set $b\Xsf'_+$ contains, among other objects,
the set $\dD(\Xsf)$ of Borel probability measures on
$\Xsf$.  For simplicity, elements of $\dD(\Xsf)$ are referred to as
\navy{distributions}.  For $\rho \in \dD(\Xsf)$ and $f \in b\Xsf$ we set
\begin{equation*}
    \inner{f, \rho} \coloneq \int f \diff \rho.
\end{equation*}

For each $x \in \Xsf$, the \navy{point evaluation functional} generated by $x$
is the map $\delta_x$ that sends $w \in b\Xsf$ into $w(x) \in \RR$.  Below
it will be convenient for us to write this in dual notation, so that
$\inner{w,\delta_x} = w(x)$ for every $w \in b\Xsf$.  We will make use of the
following lemma.

\begin{lemma}\label{l:pe}
    Every point evaluation functional on $b\Xsf$ is a nonzero element of 
    $b\Xsf_+'$.
\end{lemma}

\begin{proof}
    Fix $x \in \Xsf$. Linearity of $\delta_x$ is obvious: given $a, b \in \RR$
    and $v, w \in b\Xsf$, we have
    \begin{equation*}
        \inner{a v + b w, \delta_x}
        = (a v+ b w) (x)
        = a v(x) + b w(x) 
        = a \inner{v, \delta_x} + b \inner{w, \delta_x}.
    \end{equation*}
    Regarding continuity, if $w_n \to w$ in $b\Xsf$, then $w_n \to w$ pointwise
    on $\Xsf$, so $\inner{w_n,\delta_x} = w_n(x) \to w(x) = \inner{w,\delta_x}$. 
    Regarding positivity, it suffices to show  that $\inner{w,\delta_x} \geq 0$ 
    whenever $w \geq 0$.  This clearly holds, since $w \geq 0$ implies 
    $w(x) = \inner{w,\delta_x} \geq 0$. Finally, $\delta_x$ is not the 
    zero element of $b\Xsf'$ because 
    $w=\mathbf{1}$ is in $b\Xsf$ and $\inner{w,\delta_x} = w(x) = 1 \not= 0$.
\end{proof}

\subsection{Markov Decision Process} \label{ss:mdp}

Let $\Xsf$ and $\Asf$ be metric spaces, as in Section~\ref{ss:prel}.
A \navy{Markov decision process} (MDP) with state space $\Xsf$ and action space
$\Asf$ is a tuple $(r, \Gamma, \beta, P)$, where $r$ is a reward function,
$x \mapsto \Gamma(x) \subset \Asf$ is a feasible correspondence, $\beta \in \RR$ is a 
discount factor and $P(x, a, \diff x')$ is a distribution over next 
period states given current state $x$ and action $a$.  Let $\Gsf \coloneq
\setntn{(x, a) \in \Xsf \times \Asf}{a \in \Gamma(x)}$. We consider a relatively
standard environment, as considered in, say, \cite{bauerle2011markov} and
\cite{hernandez2012discrete}, where
\begin{enumerate}
	\item[\namedlabel{itm:a}{(a)}] $\beta \in (0, 1)$,
	\item[\namedlabel{itm:b}{(b)}] $\Gamma$ is nonempty, continuous, and
        compact-valued on $\Xsf$,
	\item[\namedlabel{itm:c}{(c)}] $r$ is bounded and continuous on $\Gsf$, and
	\item[\namedlabel{itm:d}{(d)}] the map $(x, a)
	\mapsto \int v(x') P(x, a, dx')$ is continuous on $\Gsf$ whenever $v \in
	cb\Xsf$.
\end{enumerate}
(The case of unbounded $r$ is discussed in Section~\ref{s:ext}.)

Let $\Sigma$ denote the set of feasible policies, by which we mean all Borel
measurable functions $\sigma$ mapping $\Xsf$ to $\Asf$ with $\sigma(x) \in
\Gamma(x)$ for all $x \in \Xsf$.  For each $\sigma \in \Sigma$ and $x \in \Xsf$, 
we set 
\begin{equation*}
	r_\sigma(x) \coloneq r(x, \sigma(x))
    \quad \text{and} \quad
	P_\sigma(x, \diff x') \coloneq P(x, \sigma(x), \diff x').
\end{equation*}
Thus, $r_\sigma(x)$ is rewards at $x$ under policy $\sigma$ and $P_\sigma$ is
the transition kernel on $\Xsf$ generated by $\sigma$. Using the corresponding
Markov operator $P_\sigma$, as defined in \eqref{eq:mo},  the \navy{lifetime
value} of a policy $\sigma$, denoted by $v_\sigma$, can be expressed as 
\begin{equation}\label{eq:lv}
    v_\sigma 
    = \sum_{t=0}^\infty (\beta P_\sigma)^t r_\sigma
\end{equation}
(see, e.g., \cite{puterman2014markov}, Theorem~6.1.1).
We will use the fact that $v_\sigma$ is also the unique fixed point
in $b\Xsf$ of the \navy{policy operator} $T_\sigma$ defined by $T_\sigma \, 
v = r_\sigma + \beta P_\sigma \, v$.  This operator can be written more explicitly as
\begin{equation*}
	(T_\sigma \, v)(x) = r(x, \sigma(x))
	+ \beta \int v(x') P(x, \sigma(x), \diff x')
	\qquad (v \in b\Xsf, \; x \in \Xsf).
\end{equation*}
(The lifetime value $v_\sigma$ is the unique fixed point of $T_\sigma$ because
the spectral radius of the linear operator $\beta P_\sigma$ is $\beta$, so, by the
geometric series lemma, $\beta < 1$ implies that $v = r_\sigma + \beta P_\sigma \, v$ has the unique solution given by the
right-hand side of \eqref{eq:lv}.)  Iterating on the definition $T_\sigma \, v = r_\sigma + \beta P_\sigma \, v$, we
find that 
\begin{equation}\label{eq:tsw}
    T_\sigma^n v 
    = r_\sigma + \beta P_\sigma r_\sigma
    + \cdots + (\beta P_\sigma)^{n-1} r_\sigma
    + (\beta P_\sigma)^n v
    \quad \text{for all $n \in \NN$}.  
\end{equation}
This expression will be useful in the theory below.

The \navy{value function} is denoted $v^*$ and defined at each $x \in \Xsf$ by
$v^*(x) \coloneq \sup_{\sigma \in \Sigma} v_\sigma(x)$. 
A policy $\sigma$ is called \navy{optimal} if $v_\sigma(x) = v^*(x)$ 
for all $x \in \Xsf$.

We define the Bellman operator by 
\begin{equation}\label{eq:bop}
	(Tv)(x) = 
	\max_{a \in\Gamma(x)}
	\left\{
	r(x, a) + \beta \int v(x')P(x, a, \diff x')
	\right\}
	\qquad (v \in b\Xsf, \; x \in \Xsf).
\end{equation}

We will use the following facts:

\begin{proposition}\label{p:ff}
    Under the stated assumptions,
    \begin{enumerate}
        \item the value function $v^*$ is the unique fixed point of the 
        Bellman operator in $b\Xsf$,
        \item the value function $v^*$ is well-defined and contained in
            $cb\Xsf$, and  
        \item at least one optimal policy exists.
    \end{enumerate}
\end{proposition}

\begin{proof}
    See \cite{bauerle2011markov} or \cite{hernandez2012discrete}.
\end{proof}

We also make use of the following technical lemma, which shows one implication
of local optimality at some given $x \in \Xsf$.

\begin{lemma}\label{l:e1}
	If $\sigma \in \Sigma$ and $v_\sigma(x) = v^*(x)$, then
	\begin{equation}\label{eq:pn}
	 	\sum_{n \in \NN}
        \int (v^*(x') - v_\sigma(x')) P_\sigma^n(x, \diff x') = 0.
	\end{equation}
\end{lemma}

\begin{proof}
    Fix $n \in \NN$.  Applying the expression for $T_\sigma^n v$ from
    \eqref{eq:tsw} twice, first with $v = v_\sigma$ and then with $v = v^*$, we
    get
    \begin{equation}\label{eq:tsg}
        T_\sigma^n \, v_\sigma -  T_\sigma^n \, v^* 
        =  \beta^n (P_\sigma^n \, v_\sigma -   P_\sigma^n \, v^* ).
    \end{equation}
	In addition, we have
    \begin{equation}\label{eq:vsvs}
		v_\sigma 
		= T_\sigma^n \, v_\sigma
		\leq T_\sigma^n \, v^*
		\leq T^n \, v^*
		= v^*. 
	\end{equation}
    In \eqref{eq:vsvs}, the first inequality is due to the fact $T_{\sigma}$ is order preserving 
	and $v_\sigma \leq v^* $, while the second follows from the
    fact that $T_\sigma \, v \leq Tv$ for all $v \in b\Xsf$. (Since $ T_\sigma\, 
    v \leq T\, v$ for all $v$, $ T_\sigma^2\, v^* \leq T_\sigma T\, v^*\leq T^2\, v^*$
    and so $ T_\sigma^2\, v^*\leq T^2\, v^*.$ By induction, the second 
    inequality holds.)  The claim in Lemma~\ref{l:e1} follows from \eqref{eq:tsg} and
    \eqref{eq:vsvs}.  To see this, fix $x \in \Xsf$ with $v_\sigma(x) = v^*(x)$.   From this equality and 
    \eqref{eq:vsvs} we get $(T_\sigma^n \, v_\sigma) (x) =(T_\sigma^n \, v^*)(x)$.
    Since $\beta > 0$, combining this result with \eqref{eq:tsg} yields
    $(P_\sigma^n \, v_\sigma) (x) =(P_\sigma^n \, v^*)(x)$.
    Hence \eqref{eq:pn} holds.
\end{proof}

\section{From Local to Global Optimality}\label{s:lg}

The standard definition of optimality, which was given in Section~\ref{ss:mdp},
is global in nature, since it concerns the lifetime value of the policy at every
$x \in \Xsf$. We seek conditions under which local optimality implies global
optimality.  In particular, we seek conditions under which 
the following three statements are equivalent:
\begin{enumerate}
    \item[\namedlabel{itm:E1}{(E1)}] there exists an $x \in \Xsf$ such that 
    $v_s(x)\leq v_\sigma(x)$ for all $s \in \Sigma$, 
    \item[\namedlabel{itm:E2}{(E2)}] there exists a $\rho \in \dD(\Xsf)$ such that
        $\inner{v_s, \rho}\leq\inner{ v_\sigma, \rho}$ for all 
        $s \in \Sigma$, 
    \item[\namedlabel{itm:E3}{(E3)}] $\sigma$ is an optimal policy.
\end{enumerate}

\subsection{Preliminary Results}

We first note that, in the MDP set up we have described, \ref{itm:E1} and \ref{itm:E2} are always
equivalent, as the next lemma shows.

\begin{lemma}\label{l:e1e2}
    For all $\sigma \in \Sigma$,  the statements {\rm \ref{itm:E1}}
    and {\rm \ref{itm:E2}} are equivalent.
\end{lemma}

\begin{proof}
	To show \ref{itm:E1} implies \ref{itm:E2}, assume \ref{itm:E1} 
	and fix $x \in \Xsf$  with $ v_\sigma(x)\geq v_s(x)$ for all $s\in\Sigma$.
    Let $\delta_x$ be the point evaluation functional generated by $x$.
    Since $\delta_x\in \dD(\Xsf)$ and
	$\inner{v_{\sigma},\delta_x} = v_{\sigma}(x) \geq v_s(x) = \inner{v_s,\delta_x}$
	for all $s\in \Sigma$, so \ref{itm:E2} holds. To show 
	\ref{itm:E2} implies \ref{itm:E1}, fix $\rho \in \dD(\Xsf)$ with 
	$\inner{ v_\sigma, \rho} \geq \inner{v_s, \rho} $ for all $s \in \Sigma$.
    Since at least one $s \in \Sigma$ is optimal (Proposition~\ref{p:ff}),
    this yields $\inner{ v_\sigma, \rho} \geq \inner{v^*, \rho}$.
    Now suppose to the contrary that for each $x\in\Xsf,$ we can find a $\tau \in \Sigma$
	such that $v_{\tau}(x) > v_{\sigma}(x)$. Since $v^*(x) \geq  v_s(x)$ for all 
	$s \in \Sigma$ and $x\in \Xsf$,  we have $v^*(x) > v_{\sigma}(x) $ for all 
	$x \in \Xsf$. Hence $\inner{ v^*, \rho} > \inner{v_{\sigma}, \rho}$. 
	This contradiction proves \ref{itm:E1}.
\end{proof}

Our second preliminary observation suggests one way to obtain \ref{itm:E3} from \ref{itm:E2}.

\begin{lemma}\label{l:rae}
    If $\inner{v_\sigma, \rho} = \inner{ v^*, \rho}$, then $v_\sigma = v^*$
    holds $\rho$-almost everywhere.
\end{lemma}

\begin{proof}
    By definition, $v^* \geq v_\sigma$ on $\Xsf$.  If, in addition, $v^* > v_\sigma$ on a set $E$ of positive $\rho$-measure, then
    \begin{equation*}
        \int (v^* - v_\sigma) \diff \rho
        \geq \int_E (v^* - v_\sigma) \diff \rho > 0.
    \end{equation*}
    This contradicts $\inner{v_\sigma, \rho} = \inner{ v^*, \rho}$.
    Hence $v_\sigma = v^*$ holds $\rho$-almost everywhere.
\end{proof}

Loosely speaking, this lemma shows that maximizing the scalar performance
measure $\inner{v_\sigma, \rho}$ with a highly exploratory initial distribution
$\rho$ promotes global optimality.  

The disadvantage of this approach is that evaluating the scalar measure from a
such an initial distribution is costly.  For this reason, our main
interest is in providing conditions under which \ref{itm:E1} implies \ref{itm:E3}.  Our
conditions relate to irreducibility.  The first condition, discussed in
Section~\ref{ss:si}, uses a notion of irreducibility from the literature on
Banach lattices  (see, e.g., \cite{zaanen2012introduction}), while the second,
discussed in Section~\ref{ss:wi}, uses an irreducibility concept from the Markov
process literature.

\subsection{Strong Irreducibility}\label{ss:si}

Recall that a linear subspace
$I$ of $b\Xsf$ is called an \navy{ideal} in $b\Xsf$ when $f \in I$ and $|g| \leq
|f|$ implies $g \in I$.  An ideal $I$ is said to be \navy{invariant} for a
linear operator $K$ if $K I \subset I$. A linear operator $K$ from $b\Xsf$ to
itself is called \navy{positive} when $Kf \geq 0$ for all $f \geq 0$. A positive
linear operator $K$ is called \navy{irreducible} if the only invariant ideals
under $K$ are the trivial subspace $\{0\}$ and the whole space $b\Xsf$.  
By Proposition 8.3 (c) of~\cite{schaefer1974banach}, such a $K$ is irreducible
if and only if, for each nonzero $f \in b\Xsf_+$ and each nonzero $\mu\in
b\Xsf'_+$, there exists an $m \in \NN$ with $\inner{\mu, K^m f} > 0$.

In what follows, we call a transition kernel $P$ on $\Xsf$ 
\navy{strongly irreducible} if its Markov operator (see
\eqref{eq:mo}) is irreducible on $b\Xsf$ in the sense just defined.
Here is our main result for the strongly irreducible case.
In the statement, $\sigma$ is any feasible policy.

\begin{theorem}\label{t:strong_bk}
    If $P_\sigma$ is strongly irreducible,
    then {\rm \ref{itm:E1}--\ref{itm:E3}} are equivalent.
\end{theorem}

For example, Theorem~\ref{t:strong_bk} tells us that, under the stated conditions, 
we can obtain an optimal policy by fixing an arbitrary initial state $x \in \Xsf$
and maximizing the real-valued function $s \mapsto v_s(x)$ over $\Sigma$. Alternatively, we can fix any
distribution $\rho$ and maximize $s \mapsto \inner{v_s, \rho}$.

\begin{proof}[Proof of Theorem~\ref{t:strong_bk}]
    In view of Lemma~\ref{l:e1e2}, it suffices to show that 
    \ref{itm:E1} and \ref{itm:E3} are equivalent.  That \ref{itm:E3} implies 
    \ref{itm:E1} is immediate from the definition of
    optimal policies.  Hence we need only show that \ref{itm:E1} implies \ref{itm:E3}. 
    In line with the conditions of Theorem~\ref{t:strong_bk}, we assume that 
    $\sigma$ is a feasible policy and $P_\sigma$ is strongly irreducible. 
    Using \ref{itm:E1}, we take $x \in \Xsf$ with $v_\sigma(x) = v^*(x)$.

    Let $h \coloneq v^* - v_\sigma$.  By the definition of $v^*$ we have $0 \leq h$.
    We claim in addition that $h = 0$.  To see this, suppose to the contrary that $h$ is
    nonzero. In this case, by strong irreducibility, 
    for each nonzero $\mu$ in the positive cone of $b\Xsf'$ we can find an 
    $m \in \NN$ such that $\inner{\mu, P_\sigma^m \, h} > 0$.  
    Because $\delta_{x}$ is a nonzero element of
    the positive cone of $b\Xsf'$ (Lemma~\ref{l:pe}), we can set $\mu = \delta_{x}$ to obtain
    an $m \in \NN$ with $(P_\sigma^m \, h) (x) > 0$. This contradicts
    \eqref{eq:pn}, so $h = 0$ holds.  In other words,
    $v_\sigma = v^*$.  This proves \ref{itm:E3}.
\end{proof}

\section{Topological Conditions}\label{s:tcs}

The discussion in Section~\ref{ss:sigi} shows that strong irreducibility cannot
be dropped without either (a) weakening the conclusions of
Theorems~\ref{t:strong_bk}, or (b) adding some side conditions. In this section,
we investigate both scenarios.   In particular, we show that
\begin{enumerate}
    \item even when irreducibility fails, optimality can pass across \emph{some}
        states under a continuity condition and a type of ``local irreducibility,''  and
    \item when seeking the full global conclusions of Theorem~\ref{t:strong_bk}, we
        can drop strong irreducibility if we assume a weaker form of
        irreducibility and pair it with continuity.
\end{enumerate}
The first topic is treated in Section~\ref{ss:reach}.  The second is treated in
Sections~\ref{ss:oi} and \ref{ss:wi}.

\subsection{Reachable States}\label{ss:reach}

To begin, we return to the general MDP setting from Section~\ref{ss:mdp}, where $\Xsf$ and
$\Asf$ are arbitrary metric spaces. Letting $Q$ be any stochastic kernel on
$\Xsf$, a point $y \in \Xsf$ is called \navy{$Q$-reachable} from $x \in \Xsf$
when, for each open neighborhood $G$ of $y$, there exists an $n \in \NN$ with
$Q^n(x, G) > 0$.  

\begin{lemma}\label{l:bk_access}
    Let $\sigma$ be any continuous policy.
    If $v_\sigma(x) = v^*(x)$ and $y$ is $P_\sigma$-reachable from $x$, then $v_\sigma(y) = v^*(y)$.
\end{lemma}

\begin{proof}
    As a preliminary step, we show that  $h \coloneq v^* - v_\sigma$ is
    continuous under the stated assumptions. Since $\sigma$ is continuous, our
    conditions on $(r, \Gamma, \beta, P)$ imply that the mappings $x \mapsto \int v(x')
    P(x,\sigma(x),dx')$ and $x \mapsto r(x, \sigma(x))$ are continuous on $\Xsf$ whenever $v\in cb\Xsf$. This implies
    that $T_\sigma$ is invariant on $cb\Xsf$. Moreover, $cb\Xsf$ is a closed
    subset of the complete metric space $b\Xsf$ under the supremum norm (since
    uniform limits of continuous functions are continuous). In addition, given $v,w\in b\Xsf$, we have
    \begin{align*}
     	\|T_\sigma v-T_\sigma w\|&\leq\beta\sup_{x\in\Xsf} \int |v(x')-w(x')| 
     	P(x, \sigma(x), \diff x')\\
     	&\leq\beta\sup_{x\in\Xsf} \int \|v-w\|
     	P(x, \sigma(x), \diff x') =\beta \|v-w\|.
    \end{align*}
    Since $\beta<1$, the contraction mapping theorem implies that $T_\sigma^n w
    \to v_\sigma$ for every $w \in b\Xsf$.  If we now fix $w \in cb\Xsf$ and
    use the fact that $T_\sigma$ is invariant on this set, we obtain a sequence
    $(T^n w)_{n \in \NN}$ converging to $v_\sigma$ and entirely contained in
    $cb\Xsf$.  As $cb\Xsf$ is closed in $b\Xsf$, this implies that 
    $v_\sigma$ is in $cb\Xsf$. In particular, $v_\sigma$ is continuous.  As $v^*$ 
    is also continuous (see Proposition~\ref{p:ff}), we see that $h$ is continuous. 

    Now fix $x \in \Xsf$. Seeking a contradiction, we suppose that $y$ is
    $P_\sigma$-reachable from $x$ and yet $h$ obeys
    $h(y) > 0$.   By this continuity and $h(y) > 0$, there exists an open
    neighborhood $G$ of $y$ with $h > 0$ on $G$. Because $y$ is
    $P_\sigma$-reachable from $x$, there exists an $n \in \NN$ with
    $P_\sigma^n(x, G) > 0$.  As a result, we have
    \begin{equation*}
	 	\int (v^*(x') - v_\sigma(x')) P_\sigma^n(x, \diff x') 
        \geq \int_G h(x') P_\sigma^n(x, \diff x') > 0.
    \end{equation*}
    But $v_\sigma(x) = v^*(x)$, so this inequality contradicts Lemma~\ref{l:e1}.
    The contradiction proves Lemma~\ref{l:bk_access}.
\end{proof}

\subsection{Open Set Irreducibility}\label{ss:oi}

The results in Section~\ref{ss:reach} discussed forms of ``local''
irreducibility and their implications.  In this section, we analyze settings where
these local conditions extend across the whole space and policies are
continuous.

In general, a transition kernel $Q$ from $\Xsf$ to itself is called \navy{open
set irreducible} if every $y \in \Xsf$ is reachable from every $x \in \Xsf$.
For continuous policies that generate open set irreducible transitions, we have
the following result.

\begin{theorem}\label{t:open_bk}
	If $P_\sigma$ is open set irreducible and  $\sigma$ is continuous, then
    {\rm \ref{itm:E1}--\ref{itm:E3}} are equivalent.
\end{theorem}

\begin{proof}
    In view of Lemma~\ref{l:e1e2}, it suffices to show \ref{itm:E1} implies \ref{itm:E3}.   
    So fix $x \in \Xsf$ and suppose that $v_\sigma(x) = v^*(x)$.
    For any $y \in \Xsf$, open set irreducibility implies that $y$ is $P_\sigma$-reachable from $x$. 
    Hence, by Lemma~\ref{l:bk_access}, we have $v_\sigma(y) = v^*(y)$.
    In particular, \ref{itm:E3} holds.
\end{proof}

\subsection{$\pi$-Irreducibility}\label{ss:wi}

We treat one more form of irreducibility, due to its importance in the
literature on Markov dynamics. In general, given a nontrivial measure $\pi$ on
$(\Xsf,\bB)$, a transition kernel $Q$ on $\Xsf$ is called
\navy{$\pi$-irreducible} if, for each $x \in X$ and every Borel set $B \subset
X$ with $\pi(B) > 0$, there exists an $n \in \NN$ such that $Q^n(x, B) \coloneq
(Q^n\1_B)(x) > 0$. (See, e.g., \cite{meyn2012markov}.) Here, we will say that
$Q$ is \navy{weakly irreducible} if there exists a measure $\pi$ on $(\Xsf,
\bB)$ such that
\begin{enumerate}
    \item[(ii)] $\pi$ assigns positive measure to all nonempty open sets, and
    \item[(i)] $Q$ is $\pi$-irreducible.
\end{enumerate}

\begin{lemma}\label{l:siwi}
    The following implications hold for any transition kernel $Q$ on $\Xsf$.
    \begin{enumerate}
        \item If $Q$ is strongly irreducible, then $Q$ is weakly irreducible.
        \item If $Q$ is weakly irreducible, then $Q$ is open set irreducible.
    \end{enumerate}
\end{lemma}

\begin{proof}
    Regarding (a), let $Q$ be strongly irreducible and let $\pi$ be any distribution on $\Xsf$ such
    that $\pi(G) > 0$ whenever $G \subset \Xsf$ is open and nonempty.\footnote{
    Such a measure exists in many settings, such as when
    $\Xsf$ is a locally compact topological group -- in which case we can take
    $\pi$ to be the Haar measure.  In many applications, $\Xsf$ will be a subset of
    $\RR^n$ and $\pi$ will be Lebesgue measure.} Fix $B \in \bB$ with 
    $\pi(B) > 0$ and fix $x \in  B$.
    We recall from Lemma~\ref{l:pe} that $\delta_x$ is a nonzero element of the
    dual space $b\Xsf'$.  Also, $B$ is not the empty set because $\pi(B) > 0$,
    so $\1_B$ is a nonzero element of $b\Xsf$. 
    Hence, by strong irreducibility, there exists 
    an $n \in \NN$ with $\inner{\delta_x, Q^n\1_B} > 0$.  We can rewrite this as
    $Q^n(x, B) = (Q^n\1_B)(x)>0$.  This proves that $Q$ is $\pi$-irreducible.
    We conclude that $Q$ is weakly irreducible.

    Regarding (b), let $Q$ be weakly irreducible and let $\pi$ be the measure in
    (i)--(ii) of the definition of weak irreducibility.  Pick any $x, y \in
    \Xsf$ and let $G$ be any open neighborhood of $y$.  By (i), we have $\pi(G) > 0$.
    By (ii), we can find an $m \in \NN$ with $P^m(x, G) > 0$.  Hence $y$ is
    reachable from $x$.  Since $x$ and $y$ were chosen arbitrarily, we conclude
    that $Q$ is open set irreducible.
\end{proof}

Now we state a result for the weakly irreducible case.  In the statement,
$\sigma$ is any feasible policy.

\begin{theorem}\label{t:weak_bk}
	If $P_\sigma$ is weakly irreducible and  $\sigma$ is continuous, then
    {\rm \ref{itm:E1}--\ref{itm:E3}} are equivalent.
\end{theorem}

\begin{proof}
    In view of Lemma~\ref{l:e1e2}, it suffices to show \ref{itm:E1} implies \ref{itm:E3}.   
    This is true by Theorem~\ref{t:open_bk} and Lemma~\ref{l:siwi}.
\end{proof}

\section{Applications}

In this section, we consider three applications.  The first involves an optimal
savings problem with stochastic returns on wealth, and our aim is to illustrate
the significance of our main theoretical results.  The second application  
provides an extension showing how results in the paper can be applied outside
the MDP framework.  The third application, found
in Section~\ref{ss:sigi}, demonstrates that the irreducibility assumptions in
Theorems~\ref{t:strong_bk}, \ref{t:open_bk}, and \ref{t:weak_bk} cannot be
removed without changing their conclusions.

\subsection{Optimal Savings}\label{s:os}

Consider an agent who seeks to maximize lifetime utility by choosing 
optimal savings and consumption. The evolution of wealth is governed by the equation
\begin{equation}\label{eq:wealth_flow}
    w_{t+1} = \eta_{t+1} (w_t - c_t) + y_{t+1},
    \qquad t=0, 1, \ldots,
\end{equation}
where $w_t \in \RR_+$ is time-$t$ wealth, $c_t$ is the current-period
consumption, $y_{t+1}$ is the next-period labor income, and $\eta_{t+1}$
represents the stochastic return on savings. The sequences $(y_{t})$ and
$(\eta_t)$ are IID with distributions $\phi$ and $\psi$ respectively.  For now
we assume that both of these distributions have full support on $\RR_+$. To 
simplify the notation, we use $w' = \eta'(w - c) + y'$ to denote the evolution 
of wealth.

We formulate this problem as an MDP. The state space is $\RR_+$
and the set of feasible actions at wealth level $w$ is 
 $\Gamma(w) = \{c \in \RR_+ : c \leq w\}$.
A feasible policy in this setup is a Borel measurable function $\sigma$ from
$\RR_+$ to itself satisfying $\sigma(w) \leq w$ for all $w \in \RR_+$.
The reward function is $r(w, c) \coloneq u(c)$,
where $u(c)$ is the utility derived from consumption and $u$ is
continuous, differentiable, and strictly concave on $\RR_+$. 

The transition kernel $P(w, c, \diff)$ is given by
\begin{equation}\label{eq:kernel_P}
    P(w, c, B) = \int \1_B (\eta'(w - c) + y') \psi(d\eta') \phi(dy'),
\end{equation}
where $0 \leq c \leq w$ and $B$ is a Borel set in $\RR_+$. 
Given $\sigma \in \Sigma$, the corresponding policy operator $T_\sigma$ is given
by
$$
    (T_\sigma v)(w) 
    = u(\sigma(w)) + \beta (P_\sigma \, v)(w)
    \coloneq u(\sigma(w)) + \beta \int v(w') P(w, \sigma(w), \diff w'),
$$
where $\beta \in (0, 1)$ is the discount factor. 
Using this setup, we have the following result:

\begin{lemma}\label{l:irr_P_w}
    Under the stated assumptions, the optimal savings transition kernel
	$P_\sigma$ is open set irreducible for every $\sigma \in \Sigma$.
\end{lemma}

\begin{proof}
	Let $\Sigma$ be the set of feasible policies. Fix $\bar w \in \RR_+$ and an open set
	$B \subseteq \RR_+$ with $\pi(B) > 0$. 
    Let $\alpha = w - \sigma(w)$. By a change of variable, we obtain
    \begin{align*}
        P_\sigma(\bar w,B) 
        & = \int \1_B(w') \psi(\eta') \phi(w' - \alpha \eta') 
                    \, d\eta' \, dw'
                    \\
        & = \int_B \left( \int_0^\infty \psi(\eta') \phi(w' - 
        \alpha \eta') 
        \, d\eta' \right) dw'.
    \end{align*}
    We treat the inner integral first. Since $\eta'$ and $y'$ are independent 
    random variables on $\RR_+$ with strictly positive densities $\phi$ and $\psi$ 
    respectively and $w' = \alpha \eta' + y'$, we have 
    the probability density function $f$ of $w'$ at any point $w' > 0$ is given 
    by the convolution
    $$
      f(w') = \int_0^\infty \phi(\eta') \psi(w' - \alpha \eta') d\eta'
    $$
    as in the inner integral.
    Since $w' = \eta'(w - c)+y'$ and $y' > 0$, $w' - \alpha \eta' > 0$, 
    which implies that $0 < \eta' < \frac{w'}{\alpha}$. Therefore, we have
    $$
      f(w') = \int_0^{\frac{w'}{\alpha}} \phi(\eta') \psi(w' - \alpha \eta') 
      d\eta'.
    $$
    Note that $\phi(\eta') > 0$ for $\eta' > 0$ and $\psi(w'-\alpha \eta') > 0$ 
    with $\eta' \in (0, \frac{w'}{\alpha})$. Moreover, 
    $|\frac{w'}{\alpha}| > 0$  for every $w' > 0$. Hence $f(w') > 0$ 
    for all $w' \in \RR_+$. Since $B$ is open, there is an nonempty open interval
    $(l,m)\subset B$. Thus, 
    $P_\sigma(w, B) =\int_B f(w') \; dw' \geq\int_{l}^{m} f(w') \; dw'> 0$.
     That is, $P_\sigma$ is open set irreducible.
\end{proof}

Let
\begin{equation*} 
    B(w,c,v)=u(c)+\beta\int\int v(\eta'(w-c)+y')\psi(d\eta') \phi(dy').
\end{equation*}
Since $u$ is strictly concave, the map $c\mapsto B(w,c,v)$ is strictly
concave whenever $v$ is concave on $\RR_+$.  One can also show that $v^*$ is
concave on $\RR_+$.  Combining these facts with the Bellman equation,
it is straightforward to show that the optimal policy is 
both unique and continuous.  We record this in the proposition below.
More details on the arguments can be found in Chapter~12 of~\cite{stachurski2022economic}.

\begin{proposition}\label{p:uc}
    Under the assumptions stated above, the optimal policy of the optimal
    savings model is unique and continuous on $\RR_+$.
\end{proposition}

Given the open set irreducibility of the transition kernel at any feasible policy
stated in Lemma~\ref{l:irr_P_w}, Theorem~\ref{t:open_bk} implies that we can compute a
(globally) optimal policy by maximizing $v_\sigma(\bar w)$ at any fixed $\bar w \in
\RR_+$.  We now explore this result in a computational experiment, where the
maximization of $v_\sigma(\bar w)$ is based on a simple gradient ascent algorithm
with a neural network to approximate the policy function $\sigma$.

Let $\theta\in \RR^d$ be the network parameters that 
characterize the policy network 
$\hat \sigma(\cdot; \theta)$ with a fixed paramter size $d \in \NN$.
We want to find the policy parameterization that maximizes the expected 
discounted sum of rewards defined by $ u(\hat \sigma(w; \theta))$ 
starting from the initial wealth level $\bar w$:
\begin{equation*}
    \max_{\hat \sigma \in \Sigma} v_{\hat \sigma}(\bar w) \coloneq 
    \max_{\hat \sigma \in \Sigma} 
    \sum_{t=0}^\infty 
       \beta^t \int u(\hat \sigma(w'; \theta)) P^t_{\hat \sigma}(\bar w, \diff w').
\end{equation*}
We can approximate this by Monte Carlo with large batch size $N$ 
and rollout steps $T$:
\begin{equation} \label{eq:v_hat_sigma}
    \hat v_{\hat \sigma}(\bar w) 
    \coloneq \frac{1}{N}\sum_{i=1}^N \sum_{t=0}^{T-1} 
    \beta^t u(\hat \sigma(w_{i,t}; \theta)) \quad \text{with} \quad
    w_{i,0} = \bar w \text{ for } i=1,\ldots,N, 
\end{equation}
Here $w_{i,t}$ is the wealth at time $t$ for the $i$-th sample, conditional on
starting at $\bar w$ and following policy $\hat \sigma$.
(In particular, we generate the rollout trajectories $(w_{i,t})_{t=0}^T$ for all $i=1,\ldots,N$ according to 
the wealth dynamics~\eqref{eq:wealth_flow} with consumption at time $t$  following 
$\hat \sigma(w_{i,t}; \theta)$.)
The fact that $w_{i, 0} = \bar w$ for all $i=1,\ldots,N$ means that we are
optimizing lifetime value over policies while freezing the initial state $w_0 = \bar w$. 

During each episode, the model parameters are updated to reduce the value of the loss function
\begin{equation}\label{eq:loss}
    L(\theta) =  - 
    \frac{1}{N}\sum_{i=1}^N \sum_{t=0}^{T-1} \beta^t u(\hat \sigma(w_{i,t}; \theta))
\end{equation}
using gradient descent. It is clear that minimizing the loss 
is equivalent to maximizing the finite time approximation of the expected lifetime value \eqref{eq:v_hat_sigma}.

The complete algorithm is detailed in~\algoref{alg:dpvi} in the Appendix. Our approach is 
similar to the REINFORCE algorithm~\citep{williams1992simple}, with a key distinction: since 
the utility function in our setup is differentiable, we can compute well-defined gradients 
without requiring a stochastic policy.

Several techniques could enhance the algorithm's stability and convergence, including 
reinforcement learning baselines~\citep{williams1992simple} and more advanced approaches 
applying actor-critic methods~\citep{silver2014deterministic, lillicrap2019continuous}. 
We deliberately avoid these extensions in this paper to maintain algorithmic clarity and 
simplicity, allowing us to establish a more direct connection to our theoretical results.
\begin{figure}[ht]
    \centering
    \includegraphics[width=\textwidth]{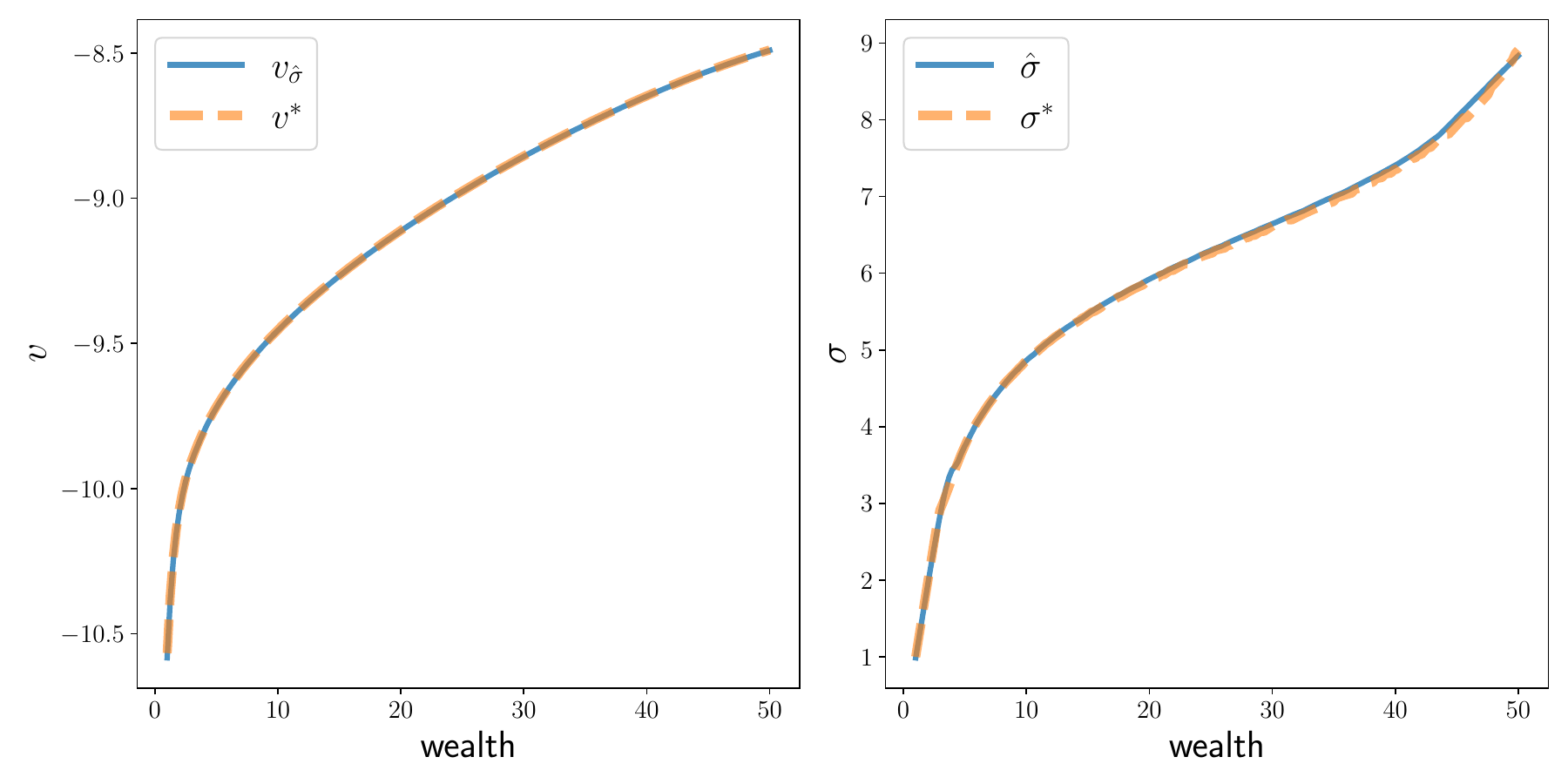}
    \caption{$v_{\hat \sigma}$ and $\hat \sigma$ with $\bar w = 1$ vs OPI solutions}
    \label{fig:dpg_vfi_irr}

    \vspace{2em}
    \centering
    \includegraphics[width=0.75\textwidth]{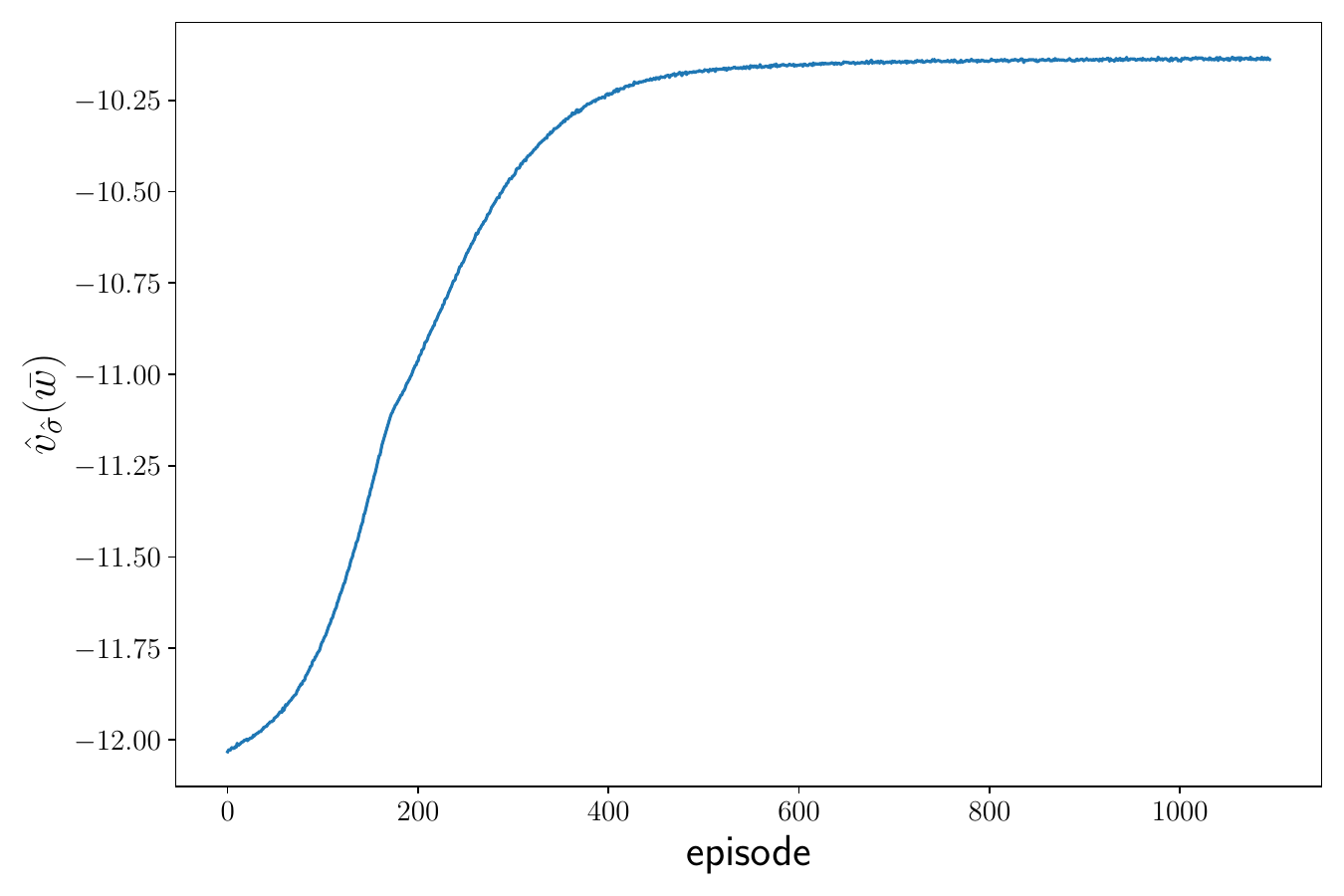}
    \caption{$\hat v_{\hat \sigma}(\bar w)$ over training episode for the 
    irreducible optimal saving model at $\bar w = 1$}
    \label{fig:irr_training}
\end{figure}
To verify our theory, we benchmark the result from optimizing the policy at 
one point against Optimistic Policy Iteration (OPI), a variant of Value Function 
Iteration (VFI) that is known to converge globally to the optimal policy in this 
model-based setting~\citep{sargent2025dynamic}. 
The OPI algorithm operates on a discretized state space and serves as our 
ground truth comparison. Specifically, we compute $v^*$ 
by applying OPI over a fine grid of wealth levels, which yields a global approximation 
of the optimal value function and corresponding optimal policy $\sigma^*$.

Regarding $v_{\hat \sigma}$, we compute this function in two steps.  First, we
freeze a single initial wealth level $\bar w$ and minimize the loss according to
\eqref{eq:loss}, as described above.  Second, we take the resulting policy $\hat \sigma$ and,
holding this policy fixed, calculate the function $v_{\hat \sigma}$ 
by computing the lifetime value of $\hat \sigma$ from alternative initial conditions.
Finally, we compare $v_{\hat \sigma}$ with the globally optimal solution $v^*$. 
In the experiments, we use the CRRA utility function $u(c) = c^{1-\gamma}/(1-\gamma)$
with $\gamma = 2$.

Figure~\ref{fig:dpg_vfi_irr} shows the result of these computations when $\bar w =
1$. The function $v_{\hat \sigma}$, shown in blue, 
closely matches the globally optimal value function $v^*$ computed 
via OPI (red dotted line). Similarly, in the second panel, 
the approximated policy $\hat \sigma$ (blue line) closely matches 
the optimal policy $\sigma^*$ (red dotted line).
This convergence demonstrates that both the $\hat \sigma$-value 
function $v_{\hat \sigma}$ and the policy $\hat \sigma$ successfully recover
their globally optimal counterparts $v^*$ and $\sigma^*$, respectively. This is
consistent with the result in Theorem~\ref{t:open_bk}, which states that we can
compute a globally optimal policy by solving $\max_{\sigma \in \Sigma} v_\sigma(\bar w)$ at any fixed
$\bar w \in \RR_+$. 

In subsequent experiments, we tested the robustness of this outcome to variation
in the fixed initial condition $\bar w$.  We found that, as predicted by theory,
the resulting function $v_{\hat \sigma}$ again closely approximates $v^*$,
and the resulting policy $\hat \sigma$ again closely approximates $\sigma^*$.
See Figure~\ref{fig:dpg_vfi_irr_50} in Appendix~\ref{app:irr_50} for a
visualization of one experiment.

In Figure~\ref{fig:irr_training}, 
the $y$-axis is the finite 
time approximation of the infinite horizon lifetime value function $v_{\hat \sigma}$
defined in~\eqref{eq:v_hat_sigma} and the $x$-axis is the number of training episodes,
which records the number of times the policy $\hat \sigma$ is updated according to 
the $\hat v_{\hat \sigma}(\bar w)$ computed from the $N$ Monte Carlo samples.
The figure shows consistent improvement in approximated lifetime value from following the 
policy $\hat \sigma$ that improves over the training episodes. The sequence of approximated 
lifetime value $\hat v_{\hat \sigma}(\bar w)$ exhibit a sharp increase in the first half 
of the training episodes, followed by a stable plateau.
\footnote{The sequence of approximated lifetime value converges 
to a higher value than $v_{\hat \sigma}(1)$. This is because the 
$\hat v_{\hat \sigma}(\bar w)$ is a finite time approximation (with rollout $T$) of the infinite 
horizon lifetime value function $v_{\hat \sigma}$. The same phenomenon occurs 
in Figure~\ref{fig:dpg_vfi_irr_50} and subsequent experiments in the reducible
case.}

\subsection{A Reducible MDP}

Now consider a modified version of the optimal saving MDP where both returns and
labor income are bounded:
\begin{equation}
    w' = \eta'(w - c) + y', \quad \eta' \in [\underline{\eta}, \overline{\eta}],
    \quad y' \in [\underline{y}, \overline{y}]
\end{equation}
with $0 < \underline{\eta} < \overline{\eta} < 1$ and $0 \leq \underline{y} <
\overline{y} < \infty$. For $w \in \RR_+$ and a Borel set $B \subseteq \RR_+$,
the stochastic kernel $P_\sigma$ is
\begin{equation}
    P_\sigma(w, B) = \int \1_B(\eta'(w - \sigma(w)) + y') \psi(d\eta') \phi(dy')
\end{equation}
where $\psi$ and $\phi$ have support on $[\underline{\eta}, \overline{\eta}]$ and
$[\underline{y}, \overline{y}]$ respectively. In this case, we let $\psi$ and
$\phi$ be uniform distributions.  For the rest of this section, we always set
$[\underline{\eta}, \overline{\eta}] = [0.5, 0.8]$ and $[\underline{y}, \overline{y}] = [1, 8]$.
At the same time, we deliberately continue to take $\RR_+$ as the state space.
This leads to a failure of irreducibility, as clarified in the next proposition.

\begin{proposition}
    For any $\sigma \in \Sigma$, the transition kernel $P_\sigma$ is neither
    weakly irreducible nor open set irreducible.
\end{proposition}

\begin{proof}
    Fix $\sigma \in \Sigma$. In view of Lemma~\ref{l:siwi}, we only 
    need to show that $P_\sigma$ is not weakly irreducible.
    To this end, let $\pi$ be a distribution on $\Wsf$ that assigns positive
    probability to open sets. Fix initial wealth $w_0 \in \RR_+$. For any $t$-step
    transition, let $\alpha_t \coloneq w_t - \sigma(w_t)$ be savings at step $t$.
    Then
    $$
        w_{t+1} \leq \overline{\eta}\alpha_t + \overline{y}
        \leq \overline{\eta}w_t + \overline{y}.
    $$
    Iterating this inequality $n$ times from $w_0$
    $$
        w_n \leq \overline{\eta}^n w_0 + \overline{y}(1 + \overline{\eta} + \dots +
        \overline{\eta}^{n-1}) = \overline{\eta}^n w_0 +
        \overline{y}\frac{1-\overline{\eta}^n}{1-\overline{\eta}}.
    $$
    Hence, there exists an $M \in \RR_+$ such that
    $$
        w_n < \overline{\eta} w_0 + \overline{y} \frac{1}{1-\overline{\eta}} < M
        \quad \forall n \in \NN.
    $$
    Let $B = (M, \infty)$. Then $\pi(B) > 0$, and
    $P^n_\sigma(w_0, B) = 0$ for all $n \in \NN$.  This shows that $P_\sigma$ is
    not weakly irreducible.
\end{proof}

\begin{figure}[ht]
    \centering
    \includegraphics[width=0.8\textwidth]{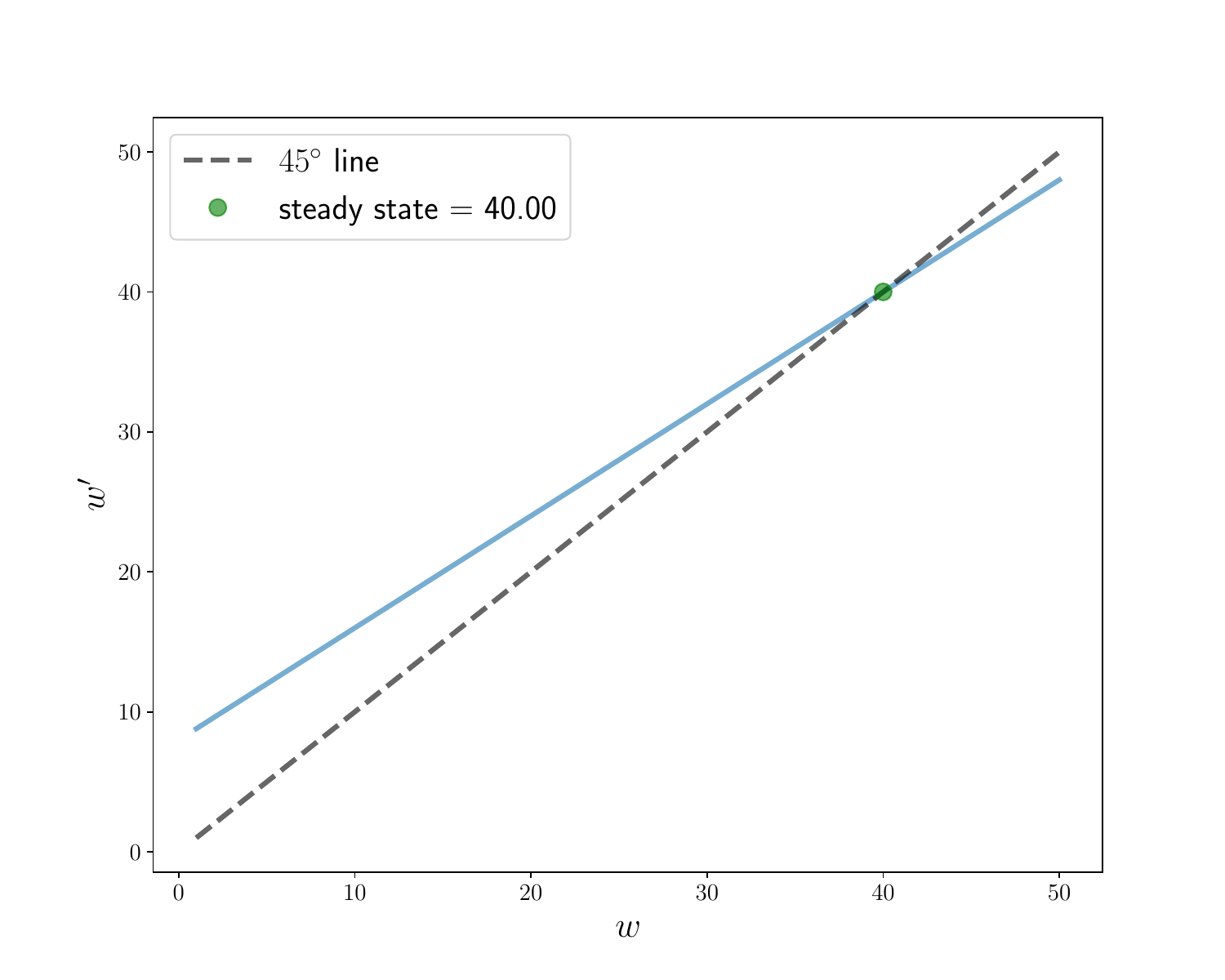}
    \caption{The upper bound law of motion for wealth}
    \label{fig:lom}
\end{figure}

The proof of Proposition~\ref{p:uc} is partly illustrated in 
Figure~\ref{fig:lom}.  The figure shows the 45 degree line and an upper bound law of motion
for wealth, obtained by setting consumption to zero and both shocks to their
upper bound.  States above the steady state are not reachable from states below
the steady state.

\begin{figure}[ht]
    \centering
    \includegraphics[width=\textwidth]{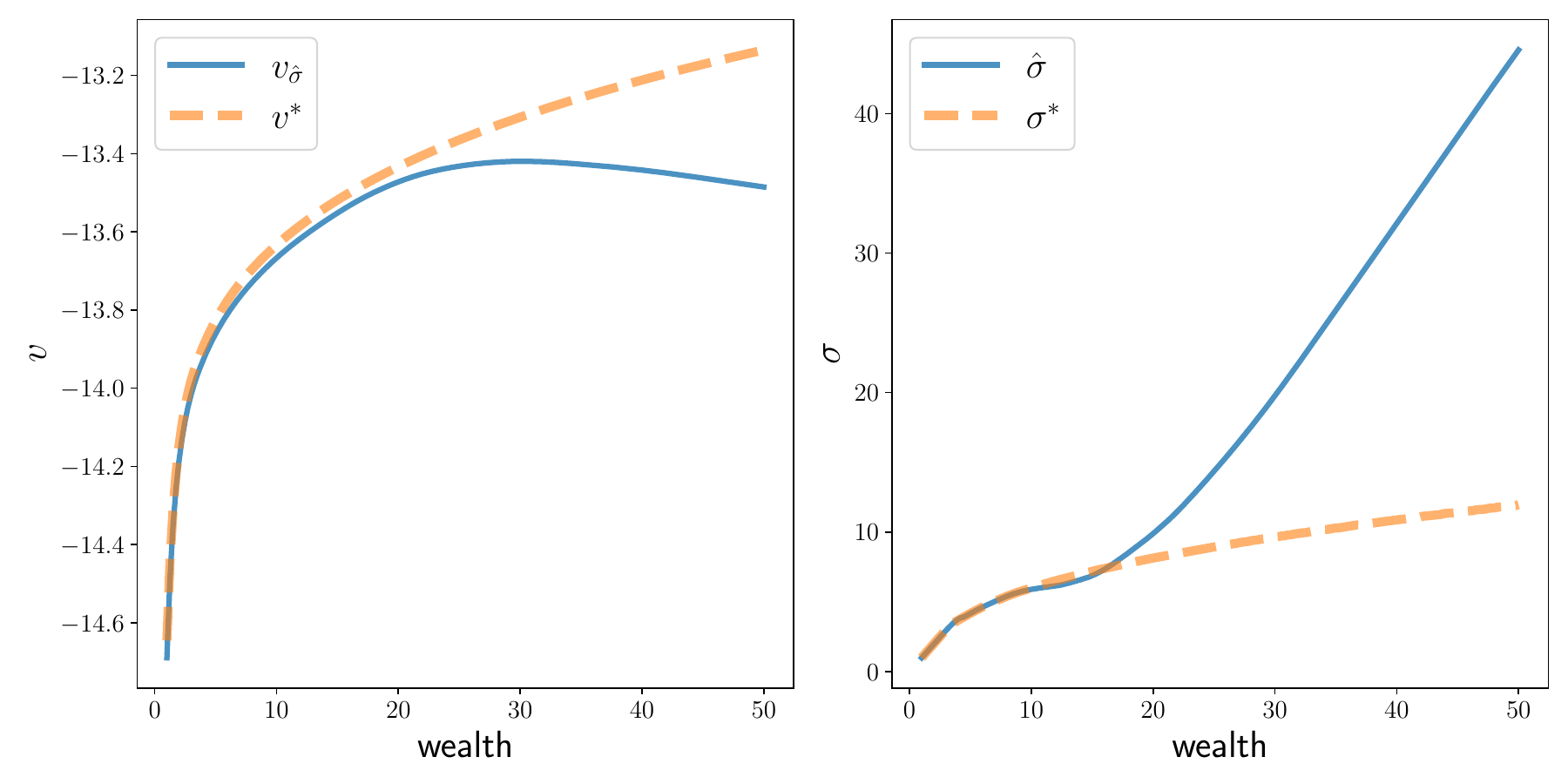}
    \caption{$\hat v_\sigma$ and $\hat \sigma$ with $\bar w = 1$ vs OPI solutions}
    \label{fig:dpg_vfi_red}

    \vspace{2em}

    \centering
    \includegraphics[width=0.75\textwidth]{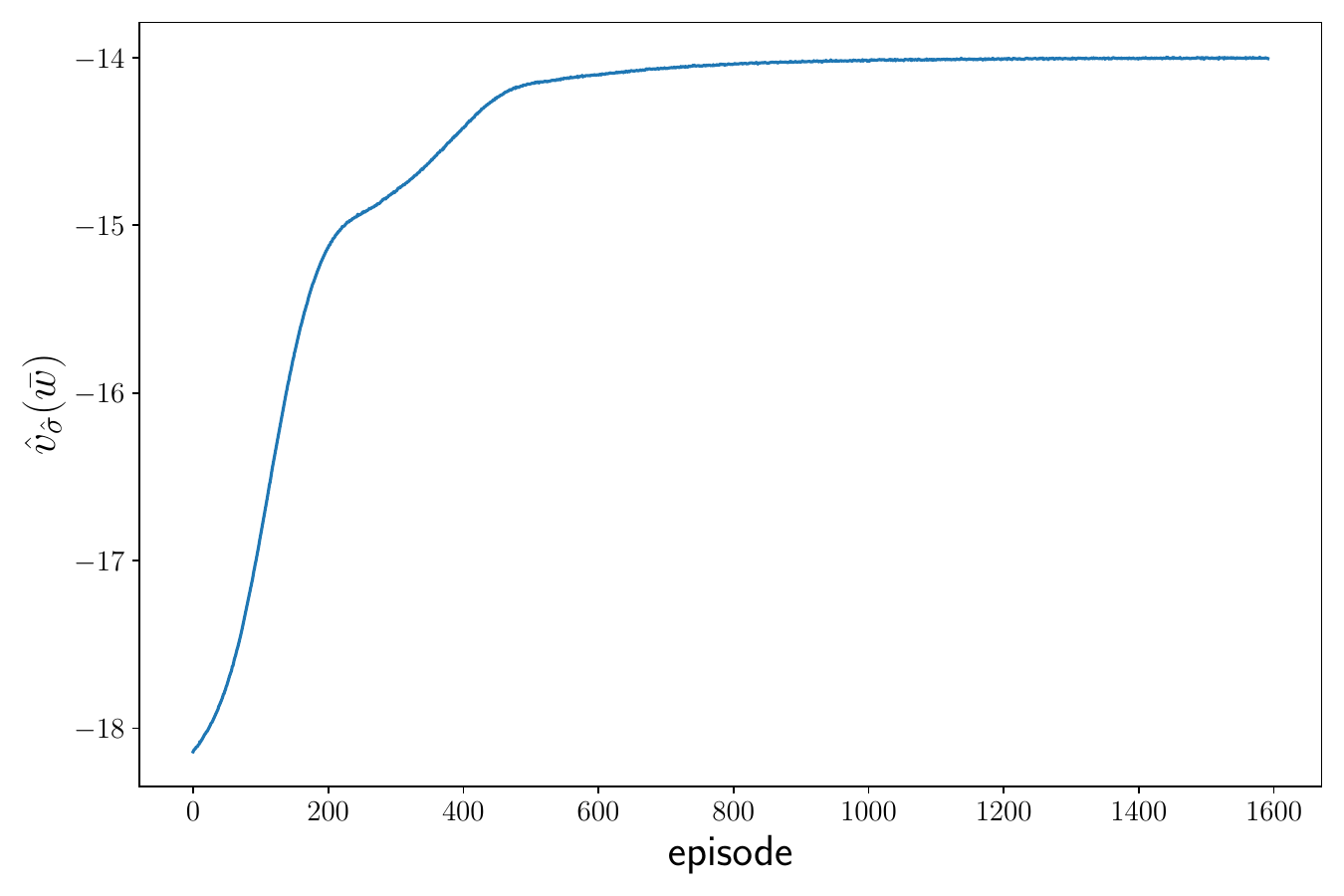}
    \caption{$\hat v_{\hat \sigma}(\bar w)$ over training episode for
    the reducible optimal saving model at $\bar w = 1$.}
    \label{fig:red_training}
\end{figure}
For reducible MDPs, Algorithm~\ref{alg:dpvi} no longer guarantees convergence to 
the optimal policy when optimized at a single initial state $\bar w \in \Wsf$. 
This limitation is clearly demonstrated in Figure~\ref{fig:dpg_vfi_red}, 
which shows significant discrepancies between the algorithm's output and the OPI 
solution. Figure~\ref{fig:red_training} confirms this limitation, 
showing that extended training episodes fail to improve the policy's 
performance.\footnote{The training episode in Figure~\ref{fig:red_training} 
extends longer than that in Figure~\ref{fig:irr_training} because we 
implement early stopping with a patience parameter of 150 episodes, meaning 
training terminates after 150 consecutive episodes without improvement in 
the loss.}

Figure~\ref{fig:dpg_vfi_red} provides empirical support for 
Lemma~\ref{l:bk_access}. The policy network achieves near-optimal 
performance at lower wealth levels because these levels are reachable 
from the initial state under the learned policy $\hat \sigma$. 
However, the performance deteriorates for wealth levels that are not reachable 
from the initial state given the policy $\hat \sigma$. (Reachable 
range of wealth level in the graph is smaller than in 
Figure~\ref{fig:lom} because $\hat \sigma(w) > 0$ for all $w \in \Wsf$.)
We also include the result for $\bar w = 50$ in Appendix~\ref{app:red_50} to 
show that the same phenomenon related to the reachable subset occurs for larger 
initial wealth levels. 

These results confirm that, when irreducibility fails, the policy $\hat \sigma$
computed to maximize local optimality is not guaranteed to be optimal, and the
performance of the policy is sensitive to the initial state and the reachable
subset of the state space starting from the initial state.

\subsection{Firm Entry}\label{ss:fen}

The next example provides an extension showing how the results in the paper can 
be used outside the MDP setting.  This is significant because many recent
examples of dynamic programs are not MDPs (see, e.g.,
\cite{sargent2025dynamic}).  The example concerns an optimal stopping problem
with state-dependent discounting.  We prove that, under the assumptions stated
below, local optimality implies global optimality.

Specifically, we consider a stopping problem for a firm choosing if and when to enter a
market.  Stopping corresponds to entering the market and continuing 
corresponds to waiting until the next period and then deciding again.  If the firm enters at time $t$ it
receives one-off profit $\pi(X_t)$, where $(X_t)_{t \geq 0}$ is a Markov
sequence on possibly unbounded interval $\Xsf \subset \RR$ with transition
kernel $Q$.  If the firm continues then it pays a fixed cost $c$. We admit the
possibility that discounting is not constant. (This allows for the fact that 
firms evaluate future profit opportunities based on cost-of-capital,
which fluctuates over time.)  Thus, the Bellman equation for the firm is
\begin{equation}\label{eq:fbe}
    v(x) = \max
    \left\{
        \pi(x) , \; 
        -c + \beta(x) \int v(x') Q(x, \diff x')
    \right\},
\end{equation}
where $\beta(x) \in (0, \infty)$ is the discount factor associated with
cost-of-capital in state $x$.  (Discount factor realizations above one are
admitted in order to accommodate occasionally negative interest rates.)

Let $ib\Xsf$ be the increasing functions in $b\Xsf$ and let $icb\Xsf$ be the
continuous functions in $ib\Xsf$.  We assume that the discount operator $K$ defined by
\begin{equation*}
    (Kf)(x) = \beta(x) \int f(x') Q(x, \diff x')
    \qquad (x \in \Xsf, \; f \in b\Xsf)
\end{equation*}
preserves monotonicity and continuity, in the sense that $K$ is invariant on
both $ib\Xsf$ and $icb\Xsf$, and that $r(K) < 1$, where
$r(K)$ is the spectral radius of $K$. For simplicity, we also assume that $Q$
always has full support on $\RR$, so that $Q(x, B) > 0$ for any Borel set $B
\subset \RR$ with
positive Lebesgue measure.  (For example, $(X_t)$ might be driven by an AR(1)
sequence with normal shocks.)  Finally, we assume that the profit function
$\pi$ is bounded, increasing, and continuous.

A policy is a map $\sigma$ from $\Xsf$ to $\{0,1\}$, with $\sigma(x)=1$
indicating the decision to stop. The lifetime value $v_\sigma$ of policy $\sigma$ 
is the unique fixed point of the operator $T_\sigma$ on $b\Xsf$ given by
\begin{equation}\label{eq:tsd}
    (T_\sigma \, v)(x) = \sigma(x) \pi(x) + (1-\sigma(x))
    \left[ -c + \beta(x) \int v(x')Q(x, \diff x') \right].
\end{equation}
We let $\Sigma$ be the set of all policies $\sigma$ such that $T_\sigma$ is
invariant on $ib\Xsf$. 

As for the MDP case, the \navy{value function} is denoted $v^*$ and defined at
each $x \in \Xsf$ by $v^*(x) \coloneq \sup_{\sigma \in \Sigma} v_\sigma(x)$. A
policy $\sigma$ is called \navy{optimal} if $v_\sigma(x) = v^*(x)$ for all $x
\in \Xsf$.  In the proofs below we will make use of the Bellman operator for
this problem, which takes the form
\begin{equation*}
    (Tv)(x) = \max
    \left\{
        \pi(x) , \; 
        -c + \beta(x) \int v(x') Q(x, \diff x')
    \right\}
    \qquad (x \in \Xsf, \; v \in b\Xsf).
\end{equation*}

The model has standard optimality properties, which we detail in the theorem
below.

\begin{theorem}\label{t:ost}
    Under the stated assumptions, the following results hold:
    \begin{enumerate}
        \item The value function is the unique solution to the Bellman equation
            in $b\Xsf$,
        \item the value function is increasing and continuous, and
        \item at least one optimal policy exists.
    \end{enumerate}
\end{theorem}

We will prove in addition the following result, which holds for any $\sigma \in
\Sigma$:

\begin{theorem}\label{t:oslbk}
    If there exists an $x \in \Xsf$ such that $v_\sigma(x) = v^*(x)$, then $\sigma$ is optimal.
\end{theorem}

We begin the proof of Theorem~\ref{t:ost} and Theorem~\ref{t:oslbk} 
with a series of lemmas.  In what follows, an operator $S$ on $b\Xsf$ is called
\navy{globally stable} when $S$ has a unique fixed point $\bar v$ in $b\Xsf$ and
$S^n v \to \bar v$ as $n\to \infty$ for all $v \in b\Xsf$.

\begin{lemma}\label{l:tsi}
    For each $\sigma \in \Sigma$, the operator $T_\sigma$ is globally stable on $b\Xsf$.
\end{lemma}

\begin{proof}
   Fix $\sigma\in \Sigma$ and $v,w\in b\Xsf$. Applying the triangle inequality,
   combined with the fact that $K$ is a positive linear operator, we have
   $|T_\sigma \, v-T_\sigma \, w| \leq |K(v-w)|\leq K|v-w|$. Since $\rho (K)<1$
   and $b\Xsf$ is a Banach lattice, by Proposition~4.1 and Theorem~2.1 in 
    \cite{stachurski2021dynamic}, there exists some $N \in \NN$ such that $T_\sigma^N$
    is a contraction. By Banach fixed point theorem, $T_\sigma^N$ is globally stable and 
    it follows that $T_\sigma$ is globally stable on $b\Xsf$. 
\end{proof}

The next proof uses several concepts from \cite{sargent2025partially}.  We refer
to that paper for the definitions, which are omitted here for brevity.

\begin{proof}[Proof of Theorem~\ref{t:ost}]
    To prove Theorem~\ref{t:ost} we apply Theorem~5.5 of
    \cite{sargent2025partially}, which applies to the abstract dynamic program
    $(b\Xsf, \TT) \coloneq (b\Xsf, \setntn{T_\sigma}{\sigma \in \Sigma})$. First, in $b\Xsf$ 
    with its usual partial order,
    the statement $v_n \uparrow v$ is equivalent to $v_n(x) \uparrow v(x)$ in
    $\RR$ for each $x \in \Xsf$.  From this fact and the monotone
	convergence theorem, we obtain $Tv_n\uparrow Tv$ whenever $(v_n)\subset
    b\Xsf$ and $v_n\uparrow v\in b\Xsf$. Thus, $(b\Xsf, \TT)$ is order continuous.  
    Also, $b\Xsf$ is countably-Dedekind
    complete because the set of bounded Borel measurable functions is closed
    under pointwise convergence.  In addition, each $T_\sigma$ is globally
    stable (Lemma~\ref{l:tsi}) and hence order stable (see \cite{sargent2025partially}, Example~2.1).
    Finally, let $\hat T$ be the operator given by $\hat T v= |\pi| +Kv$. 
    Since $\rho(K)<1$, there is a $u\in b\Xsf_+$ with $\hat T u=u$. 
    For this $u$ we have $T_\sigma \, u \leq \hat T u \leq u$, so $(b\Xsf, \TT)$ is
    bounded above.  It now follows from Theorem~5.5 of \cite{sargent2025partially}
    that the value function is the unique solution
    to the Bellman equation and at least one optimal policy exists. 
	
    To prove the final claim we first note that $T$ is globally stable on
    $b\Xsf$, since, given $v, w \in b\Xsf$, we have
    $|T v-T w| \leq |K(v-w)|\leq K|v-w|$, so the argument in the proof of
    Lemma~\ref{l:tsi} applies. Now fix $v\in icb\Xsf$. Since $K$ is invariant on
    $icb\Xsf$,  the element $-c+Kv$ is in $icb\Xsf$. Moreover, $\pi\in icb\Xsf$ 
    and so $\max\{\pi,-c+Kv\} \in icb\Xsf$. This shows that $T$ is invariant on $icb\Xsf$.  
    Since $icb\Xsf$ is nonempty and closed under uniform limits, and $T$ is globally 
    stable on $b\Xsf$, we conclude that $v^*\in icb\Xsf$.  The proof of 
    Theorem~\ref{t:ost} is now done.
\end{proof}

\begin{lemma}\label{l:vmx}
    If $v_\sigma(x) = v^*(x)$ for some 
    $x \in \Xsf$, then $v_\sigma = v^*$ almost everywhere.
\end{lemma}

\begin{proof}
     Since $v_\sigma \leq v^*$ and $T_\sigma \, v \leq T v$ for all $v \in b\Xsf$,
     we have $v_\sigma = T_\sigma \, v_\sigma \leq T_\sigma \, v^* \leq T v^* =
     v^*$. As a result, $v_\sigma(x) = v^*(x)$ implies 
     $(T_\sigma \, v_\sigma)(x) = (T_\sigma \, v^*)(x)$, which, using the 
     definition of $T_\sigma$ in \eqref{eq:tsd}, yields
     \begin{equation*}
         \beta(x) \int v^*(x') Q(x, \diff x')
         = \beta(x) \int v_\sigma(x') Q(x, \diff x').
     \end{equation*}
     As $\beta(x) > 0$, we obtain $\int (v^*(x') - v_\sigma(x')) Q(x, \diff x')
     = 0$.  If $v^* > v_\sigma$ on a set $E$ of positive Lebesgue measure, then,
     at the same time, we have
     \begin{equation*}
         \int (v^*(x') - v_\sigma(x')) Q(x, \diff x') 
         \geq \int_E (v^*(x') - v_\sigma(x')) Q(x, \diff x') > 0,
     \end{equation*}
     where the last inequality is due to the assumption that $Q(x, \diff x')$
     assigns positive measure to any such $E$.  From this contradiction, we
     conclude that $v^* = v_\sigma$ almost everywhere.
\end{proof}

\begin{lemma}\label{l:cnc}
    Let $f, g$ be two increasing functions in $b\Xsf$ with $f \leq g$. 
    If $g$ is continuous and $f=g$ almost everywhere, then $f=g$.
\end{lemma}

\begin{proof}
    If $f$ is continuous, the result follows from the fact that any two continuous measurable 
    functions equal almost everywhere on an
    interval of $\RR$ are equal. Now suppose instead that $f$ is not
    continous at $x_0 \in \Xsf$. Let $a=\lim_{x\uparrow x_0} f(x)$ and 
    $b=\lim_{x\downarrow x_0} f(x)$. 
    Monotonicity and discontinuiy of $f$ implies $a<b$. Since $g$ is continuous,  
    $g((a,b))^{-1}$ is open. We claim that $g((a,b))^{-1}$ is nonempty. 
    Indeed, since $f=g$ almost everywhere, there are $y,y' \in \Xsf$ such that $g(y)\leq a $ 
    and $g(y')\geq b$.  Since $g$ is continuous and increasing, there exist 
    $y'' \in \Xsf$ such that $g(y'')\in (a,b)$. Thus $g((a,b))^{-1}$ is 
    a nonempty open set.  This proves the existence of a set of nonzero Lebesgue
    measure on which $f < g$. Contradiction.
\end{proof}

\begin{proof}[Proof of Theorem~\ref{t:oslbk}]
    Suppose there exists an $x \in \Xsf$ and a policy $\sigma$ such that $v_\sigma(x)
    = v^*(x)$.  Then, by Lemma~\ref{l:vmx}, we have $v_\sigma = v^*$ almost everywhere.
    By Theorem~\ref{t:oslbk}, the function $v^*$ is increasing and continuous.
    By Lemma~\ref{l:tsi}, the operator $T_\sigma$ is invariant on the closed set
    $ib\Xsf$, so the fixed point is in $ib\Xsf$; in particular, $v_\sigma$ is
    also increasing.  By definition, $v_\sigma \leq v^*$.  Hence, by
    Lemma~\ref{l:cnc}, we have $v_\sigma = v^*$.
\end{proof}

\subsection{A Two-State Example}\label{ss:sigi}

In this section,  we show that the irreducibility assumptions used in
Theorem~\ref{t:strong_bk}, Theorem~\ref{t:open_bk}, and Theorem~\ref{t:weak_bk}
cannot be dropped: without irreducibility, \ref{itm:E1}--\ref{itm:E3} are not
generally equivalent. To make the example as simple as possible, we specialize
to the case where $\Xsf$ is finite and study the role of irreducibility in this
setting. In applying these theorems, we take the metric on $\Xsf$ to be the
discrete metric, so that all subsets of $\Xsf$ are open and $\bB$ is all subsets
of $\Xsf$.   We also assume that $\Asf$ is finite and impose the discrete metric
on $\Asf$.  As a result, every map and hence every policy from $\Xsf$ to $\Asf$
is continuous.

Prior to stating our example, we note some elementary properties of
irreducibility on discrete states. Given a transition kernel $Q$ on $\Xsf$, we say that
$y$ in $\Xsf$ is \navy{$Q$-accessible} from $x \in\Xsf$ when there exists an $m
\in \NN$ such that $Q^m(x,y) > 0$.   The usual definition of irreducibility of a
transition kernel $Q$ on finite $\Xsf$ is that, for every $x, y \in \Xsf$, $x$
is $Q$-accessible from $y$ and $y$ is $Q$-accessible from $x$. To distinguish
between different notions of irreducibility, we call this \navy{discrete
irreducibility}.

\begin{lemma}\label{l:dc}
    Let $\sigma$ be any feasible policy. When the state space is finite, the following
    statements are equivalent:
    \begin{enumerate}
        \item $P_\sigma$ is strongly irreducible.
        \item $P_\sigma$ is discretely irreducible.
    \end{enumerate}
\end{lemma}

\begin{proof}
    ((a) $\implies$ (b)) Fix $x, y \in \Xsf$.  Let $\1_x(z)$ equal $1$ when $z=x$ and
    zero elsewhere.  Let $\1_y$ be defined analogously.  Since $\1_x \in b\Xsf_+$ and 
    $\1_y \in (b\Xsf_+)'$, there exists an $m \in \NN$ with 
    $\inner{\1_x, P_\sigma^m \1_y} > 0$. But 
    $\inner{\1_x, P_\sigma^m \1_y} = P_\sigma^m(x, y)$, so $y$ is
    $P_\sigma$-accessible from $x$.  This proves that $P_\sigma$ is discretely
    irreducible.

    ((b) $\implies$ (a)) Fix a nonzero $f \in b\Xsf_+$ and nonzero 
    $\mu\in (b\Xsf_+)'$.  Since $\Xsf$ is finite, these element are just
    maps from $\Xsf$ to $\RR_+$ and, as both are nonzero, we can find $\bar x,
    \bar y \in \Xsf$ such that $f(\bar y)>0$ and $\mu(\bar x)>0$.  Moreover,
    since $P_\sigma$ is discretely irreducible, there is $m\in\NN$, such that
    $P_\sigma^m(\bar x,\bar y)>0$.  As a result,
    \begin{equation*}
        \inner{\mu, P_\sigma^m f}
        = \sum_x ( P_\sigma^m f)(x) \mu(x)
        = \sum_x \sum_y f(y) P_\sigma^m (x, y) \mu(x)
        \geq P_\sigma^m(\bar x, \bar y) f(\bar y) \mu(\bar x) >0.
    \end{equation*}
    This proves that $P_\sigma$ is strongly irreducible.
\end{proof}

Now consider a two-state MDP with $\Xsf = \{1, 2\}$ and $\Asf = \{1, 2\}$. The
feasible coprrespondence is defined by $\Gamma(1) = \{1, 2\}$ and $\Gamma(2) =
\{2\}$. The reward function is defined by $r(i,j) = r_{ij}$ with
\begin{equation*}
    \begin{pmatrix}
        r_{11} & r_{12} \\
        r_{21} & r_{22} 
    \end{pmatrix}
    =
    \begin{pmatrix}
        0 & 1 \\
        0 & 2
    \end{pmatrix}.
\end{equation*}
We set $\beta = 0.9$. The transition probabilities $P(x,a,x')$ are given by
\begin{align*} 
P(1, 1, \cdot) = (1, 0), \quad P(1, 2, \cdot) = (1, 0), 
\quad P(2, 1, \cdot) = (0, 1), \quad P(2, 2, \cdot) = (0, 1).
\end{align*}
By the definition of the feasible correspondence $\Gamma$, there are only two 
feasible policies $\Sigma = \{\sigma, \pi\}$, where $\sigma(x) = 2$ and 
$\pi(x) = x$ for all $x \in \Xsf$. The transition probabilities following the 
two policies are given by
\begin{equation*}
    P_\sigma = P_\pi =
    \begin{pmatrix}
        1 & 0 \\
        0 & 1
    \end{pmatrix}.
\end{equation*}
Observe that $P_\sigma = P_\pi$ fails to be discretely irreducible.  By
Lemma~\ref{l:siwi} and Lemma~\ref{l:dc}, this means that this transition
kernel is neither strongly irreducible, nor weakly irreducible, nor open set irreducible.

Let us compute the lifetime value functions for the optimal $\sigma$ and $\pi$.
For policy $\sigma$, we have $r_\sigma = (r_{12}, r_{22}) = (1, 2)$ and
$v_\sigma = (10, 20)$. For $\pi$, we have $r_\pi = (r_{11}, r_{22}) = (0, 2)$
and $v_\pi = (0, 20)$. Since, $v^*(x) \coloneq \sup_{s \in \Sigma} v_s(x) =
v_\sigma(x)$, we see that $\sigma$ is an optimal policy. On the other hand,
$v_\pi(2) = v^*(2)$, but $v_\pi(1) < v^*(1)$, indicating that optimality in one
state does not guarantee global optimality.  This confirms that irreducibility
cannot be dropped from the statements of Theorems~\ref{t:strong_bk},
\ref{t:open_bk}, and \ref{t:weak_bk}.

\section{Extensions and Future Work}\label{s:ext}

Using MDP optimality results from \cite{bauerle2011markov} or
\cite{bertsekas2022abstract}, it should be possible to extend our results to the
case of unbounded rewards by replacing the ordinary supremum norm on $b\Xsf$
with a weighted supremum norm.  Also, while our results have focused on standard
MDPs with constant discount factors, one useful variation of this model is MDPs
with state-dependent discount factors, so that $\beta$ becomes a map from $\Xsf$
to $\RR_+$ (see, e.g., \cite{stachurski2021dynamic}).  We conjecture that
the main results we obtained for MDPs will extend to generalized ADPs
with state-dependent discounting under suitable stability and irreducibility
assumptions.   Finally, it seems likely that results similar to Theorem~\ref{t:strong_bk} will 
be valid for some continuous time MDPs, as well as at least some of the
nonstandard discrete time dynamic programs discussed in \cite{bertsekas2022abstract} 
and \cite{sargent2025dynamic}. These topics are also left for future work.

\newpage

\bibliographystyle{ecta}
\bibliography{localbib}

\begin{thebibliography}{32}
\newcommand{\enquote}[1]{``#1''}
\expandafter\ifx\csname natexlab\endcsname\relax\def\natexlab#1{#1}\fi

\bibitem[\protect\citeauthoryear{Agarwal, Kakade, Lee, and Mahajan}{Agarwal et~al.}{2021}]{agarwal2021theory}
\textsc{Agarwal, A., S.~M. Kakade, J.~D. Lee, and G.~Mahajan} (2021): \enquote{On the theory of policy gradient methods: Optimality, approximation, and distribution shift,} \emph{Journal of Machine Learning Research}, 22, 1--76.

\bibitem[\protect\citeauthoryear{Bartlett and Tewari}{Bartlett and Tewari}{2009}]{Bartlett2009REGAL}
\textsc{Bartlett, P.~L. and A.~Tewari} (2009): \enquote{REGAL: a regularization based algorithm for reinforcement learning in weakly communicating MDPs,} in \emph{Proceedings of the Twenty-Fifth Conference on Uncertainty in Artificial Intelligence}, AUAI Press, 35–42.

\bibitem[\protect\citeauthoryear{B{\"a}uerle and Rieder}{B{\"a}uerle and Rieder}{2011}]{bauerle2011markov}
\textsc{B{\"a}uerle, N. and U.~Rieder} (2011): \emph{Markov decision processes with applications to finance}, Springer Science \& Business Media.

\bibitem[\protect\citeauthoryear{Bertsekas}{Bertsekas}{2012}]{bertsekas2012dynamic}
\textsc{Bertsekas, D.} (2012): \emph{Dynamic programming and optimal control}, vol.~1, Athena Scientific.

\bibitem[\protect\citeauthoryear{Bertsekas}{Bertsekas}{2021}]{bertsekas2021rollout}
---\hspace{-.1pt}---\hspace{-.1pt}--- (2021): \emph{Rollout, policy iteration, and distributed reinforcement learning}, Athena Scientific.

\bibitem[\protect\citeauthoryear{Bertsekas}{Bertsekas}{2022}]{bertsekas2022abstract}
\textsc{Bertsekas, D.~P.} (2022): \emph{Abstract dynamic programming}, Athena Scientific, 3 ed.

\bibitem[\protect\citeauthoryear{Bhandari and Russo}{Bhandari and Russo}{2024}]{bhandari2024global}
\textsc{Bhandari, J. and D.~Russo} (2024): \enquote{Global optimality guarantees for policy gradient methods,} \emph{Operations Research}.

\bibitem[\protect\citeauthoryear{Friedl, K{\"u}bler, Scheidegger, and Usui}{Friedl et~al.}{2023}]{friedl2023deep}
\textsc{Friedl, A., F.~K{\"u}bler, S.~Scheidegger, and T.~Usui} (2023): \enquote{Deep uncertainty quantification: with an application to integrated assessment models,} Tech. rep., Working Paper University of Lausanne.

\bibitem[\protect\citeauthoryear{Hern{\'a}ndez-Lerma and Lasserre}{Hern{\'a}ndez-Lerma and Lasserre}{2012}]{hernandez2012discrete}
\textsc{Hern{\'a}ndez-Lerma, O. and J.~B. Lasserre} (2012): \emph{Discrete-time {M}arkov control processes: basic optimality criteria}, vol.~30, Springer Science \& Business Media.

\bibitem[\protect\citeauthoryear{Khodadadian, Jhunjhunwala, Varma, and Maguluri}{Khodadadian et~al.}{2021}]{khodadadian2021linear}
\textsc{Khodadadian, S., P.~R. Jhunjhunwala, S.~M. Varma, and S.~T. Maguluri} (2021): \enquote{On the Linear Convergence of Natural Policy Gradient Algorithm,} \emph{2021 60th IEEE Conference on Decision and Control (CDC)}, 3794--3799.

\bibitem[\protect\citeauthoryear{Kochenderfer, Wheeler, and Wray}{Kochenderfer et~al.}{2022}]{kochenderfer2022algorithms}
\textsc{Kochenderfer, M.~J., T.~A. Wheeler, and K.~H. Wray} (2022): \emph{Algorithms for decision making}, The MIT Press.

\bibitem[\protect\citeauthoryear{Kumar, Derman, Geist, Levy, and Mannor}{Kumar et~al.}{2023}]{Kumar2023PolicyGF}
\textsc{Kumar, N., E.~Derman, M.~Geist, K.~Y. Levy, and S.~Mannor} (2023): \enquote{Policy Gradient for Rectangular Robust Markov Decision Processes,} in \emph{Neural Information Processing Systems}.

\bibitem[\protect\citeauthoryear{Lan, Wang, Anderson, Brinton, and Aggarwal}{Lan et~al.}{2023}]{Lan2023ImprovedCE}
\textsc{Lan, G., H.~Wang, J.~Anderson, C.~G. Brinton, and V.~Aggarwal} (2023): \enquote{Improved Communication Efficiency in Federated Natural Policy Gradient via ADMM-based Gradient Updates,} \emph{ArXiv}, abs/2310.19807.

\bibitem[\protect\citeauthoryear{Lillicrap, Hunt, Pritzel, Heess, Erez, Tassa, Silver, and Wierstra}{Lillicrap et~al.}{2019}]{lillicrap2019continuous}
\textsc{Lillicrap, T.~P., J.~J. Hunt, A.~Pritzel, N.~Heess, T.~Erez, Y.~Tassa, D.~Silver, and D.~Wierstra} (2019): \enquote{Continuous control with deep reinforcement learning,} .

\bibitem[\protect\citeauthoryear{Maliar, Maliar, and Winant}{Maliar et~al.}{2021}]{maliar2021deep}
\textsc{Maliar, L., S.~Maliar, and P.~Winant} (2021): \enquote{Deep learning for solving dynamic economic models.} \emph{Journal of Monetary Economics}, 122, 76--101.

\bibitem[\protect\citeauthoryear{Meyn and Tweedie}{Meyn and Tweedie}{2012}]{meyn2012markov}
\textsc{Meyn, S.~P. and R.~L. Tweedie} (2012): \emph{Markov chains and stochastic stability}, Springer Science \& Business Media.

\bibitem[\protect\citeauthoryear{Mnih, Badia, Mirza, Graves, Lillicrap, Harley, Silver, and Kavukcuoglu}{Mnih et~al.}{2016}]{mnih2016asynchronous}
\textsc{Mnih, V., A.~P. Badia, M.~Mirza, A.~Graves, T.~P. Lillicrap, T.~Harley, D.~Silver, and K.~Kavukcuoglu} (2016): \enquote{Asynchronous Methods for Deep Reinforcement Learning,} .

\bibitem[\protect\citeauthoryear{Murphy}{Murphy}{2024}]{murphy2024reinforcementlearningoverview}
\textsc{Murphy, K.} (2024): \enquote{Reinforcement Learning: An Overview,} .

\bibitem[\protect\citeauthoryear{Paszke, Gross, Chintala, Chanan, Yang, DeVito, Lin, Desmaison, Antiga, and Lerer}{Paszke et~al.}{2017}]{paszke2017automatic}
\textsc{Paszke, A., S.~Gross, S.~Chintala, G.~Chanan, E.~Yang, Z.~DeVito, Z.~Lin, A.~Desmaison, L.~Antiga, and A.~Lerer} (2017): \enquote{Automatic differentiation in pytorch,} .

\bibitem[\protect\citeauthoryear{Puterman}{Puterman}{2014}]{puterman2014markov}
\textsc{Puterman, M.~L.} (2014): \emph{Markov decision processes: discrete stochastic dynamic programming}, John Wiley \& Sons.

\bibitem[\protect\citeauthoryear{Sargent and Stachurski}{Sargent and Stachurski}{2025{\natexlab{a}}}]{sargent2025dynamic}
\textsc{Sargent, T.~J. and J.~Stachurski} (2025{\natexlab{a}}): \emph{Dynamic Programming: Finite States}, Cambridge University Press.

\bibitem[\protect\citeauthoryear{Sargent and Stachurski}{Sargent and Stachurski}{2025{\natexlab{b}}}]{sargent2025partially}
---\hspace{-.1pt}---\hspace{-.1pt}--- (2025{\natexlab{b}}): \enquote{Dynamic Programs on Partially Ordered Sets,} \emph{SIAM Journal on Control and Optimization}, 63, 778--795.

\bibitem[\protect\citeauthoryear{Schaefer}{Schaefer}{1974}]{schaefer1974banach}
\textsc{Schaefer, H.~H.} (1974): \emph{Banach Lattices and Positive Operators}, Springer.

\bibitem[\protect\citeauthoryear{Schulman, Levine, Abbeel, Jordan, and Moritz}{Schulman et~al.}{2015}]{Schulman2015TRPO}
\textsc{Schulman, J., S.~Levine, P.~Abbeel, M.~Jordan, and P.~Moritz} (2015): \enquote{Trust Region Policy Optimization,} in \emph{Proceedings of the 32nd International Conference on Machine Learning}, PMLR.

\bibitem[\protect\citeauthoryear{Schulman, Wolski, Dhariwal, Radford, and Klimov}{Schulman et~al.}{2017}]{schulman2017proximal}
\textsc{Schulman, J., F.~Wolski, P.~Dhariwal, A.~Radford, and O.~Klimov} (2017): \enquote{Proximal policy optimization algorithms,} \emph{arXiv preprint arXiv:1707.06347}.

\bibitem[\protect\citeauthoryear{Silver, Lever, Heess, Degris, Wierstra, and Riedmiller}{Silver et~al.}{2014}]{silver2014deterministic}
\textsc{Silver, D., G.~Lever, N.~Heess, T.~Degris, D.~Wierstra, and M.~Riedmiller} (2014): \enquote{Deterministic policy gradient algorithms,} in \emph{International conference on machine learning}, Pmlr, 387--395.

\bibitem[\protect\citeauthoryear{Stachurski}{Stachurski}{2022}]{stachurski2022economic}
\textsc{Stachurski, J.} (2022): \emph{Economic dynamics: theory and computation}, MIT Press, 2 ed.

\bibitem[\protect\citeauthoryear{Stachurski and Zhang}{Stachurski and Zhang}{2021}]{stachurski2021dynamic}
\textsc{Stachurski, J. and J.~Zhang} (2021): \enquote{Dynamic programming with state-dependent discounting,} \emph{Journal of Economic Theory}, 192, 105190.

\bibitem[\protect\citeauthoryear{Sutton, McAllester, Singh, and Mansour}{Sutton et~al.}{1999}]{Sutton1999PolicyGM}
\textsc{Sutton, R.~S., D.~A. McAllester, S.~Singh, and Y.~Mansour} (1999): \enquote{Policy Gradient Methods for Reinforcement Learning with Function Approximation,} in \emph{Neural Information Processing Systems}.

\bibitem[\protect\citeauthoryear{Williams}{Williams}{1992}]{williams1992simple}
\textsc{Williams, R.~J.} (1992): \enquote{Simple statistical gradient-following algorithms for connectionist reinforcement learning,} \emph{Machine learning}, 8, 229--256.

\bibitem[\protect\citeauthoryear{Xiao}{Xiao}{2022}]{xiao2022convergence}
\textsc{Xiao, L.} (2022): \enquote{On the convergence rates of policy gradient methods,} \emph{Journal of Machine Learning Research}, 23, 1--36.

\bibitem[\protect\citeauthoryear{Zaanen}{Zaanen}{2012}]{zaanen2012introduction}
\textsc{Zaanen, A.~C.} (2012): \emph{Introduction to operator theory in Riesz spaces}, Springer.

\end{thebibliography}

\newpage

\appendix

\section{Algorithm}
This section provides a detailed description of the algorithm used to train the 
policy network described in Section~\ref{s:os}. In the state transition step, the algorithm uses the wealth 
bounds $(w_{\min}, w_{\max})$ to ensure that the wealth process remains within 
these bounds to match the OPI solution, which operates on a wealth grid.
In practice, we can replace it with $w_{t+1} \gets \eta_{t} (w_t - c) + y_{t}$. 

We implement the policy network and optimizer using PyTorch. 
The training is conducted on an NVIDIA RTX 5070 Ti GPU with CUDA 12.8. 
Our implementation uses Python 3.13.2 and leverages PyTorch's built-in automatic 
differentiation \citep{paszke2017automatic} to compute gradients of 
the loss function efficiently.

\begin{algorithm}[H]
    \caption{\label{alg:dpvi}Deterministic Policy Gradient}
    \small
    \DontPrintSemicolon
    \KwIn{Discount factor $\beta$, distributions $(\psi, \phi)$, episodes $K$, rollout steps $T$, 
    batch size $N$, learning rate $\alpha$, wealth bounds $(w_{\min}, w_{\max})$}
    Initialize policy network $\hat \sigma$ and optimizer\;
    \For{episode $k = 1$ to $K$}{
        Initialize states $w \gets \bar{w} \in \RR^N$ \; 
        Pre-sample shocks vectors $\eta_{t} \sim \psi$, $y_{t} \sim \phi$ for all $t = 0, 1, \ldots, T-1$\;
        \For{$t = 0$ to $T-1$}{
            $c \gets \hat \sigma(w_t; \theta)$ \; 
            $w_{t+1} \gets \min\{\max\{\eta_{t} (w_t - c) + y_{t}, w_{\min}\}, w_{\max}\}$ \tcp*{Transitions}
        }
        $\theta \gets \theta - \alpha \nabla_\theta L(\theta) = 
        \theta + \alpha \nabla_\theta \frac{1}{N}\sum_{i=1}^N \sum_{t=0}^{T-1} \beta^t u(\hat \sigma(w_{i,t}; \theta))$\;
    }
    \KwRet{$\hat \sigma(\cdot; \theta)$}
\end{algorithm}

\newpage 

\section{Irreducible Optimal Savings MDP with $\bar w = 50$} \label{app:irr_50}
\begin{figure}[ht]
    \centering
    \includegraphics[width=\textwidth]{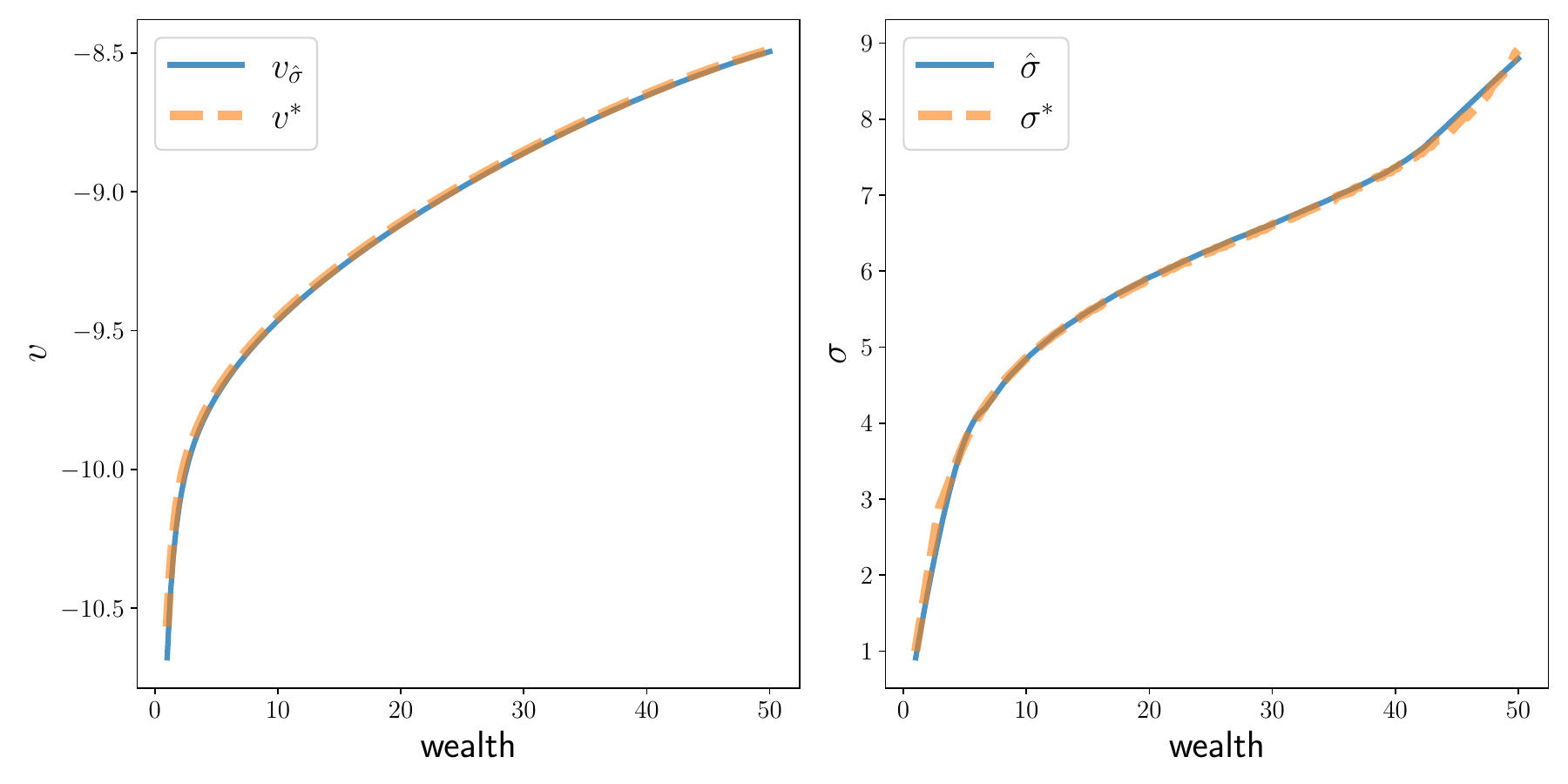}
    \caption{$\hat v_\sigma$ and $\hat \sigma$ with $\bar w = 50$ against
    the OPI solutions.}
    \label{fig:dpg_vfi_irr_50}

    \vspace{2em}
    \centering
    \includegraphics[width=0.75\textwidth]{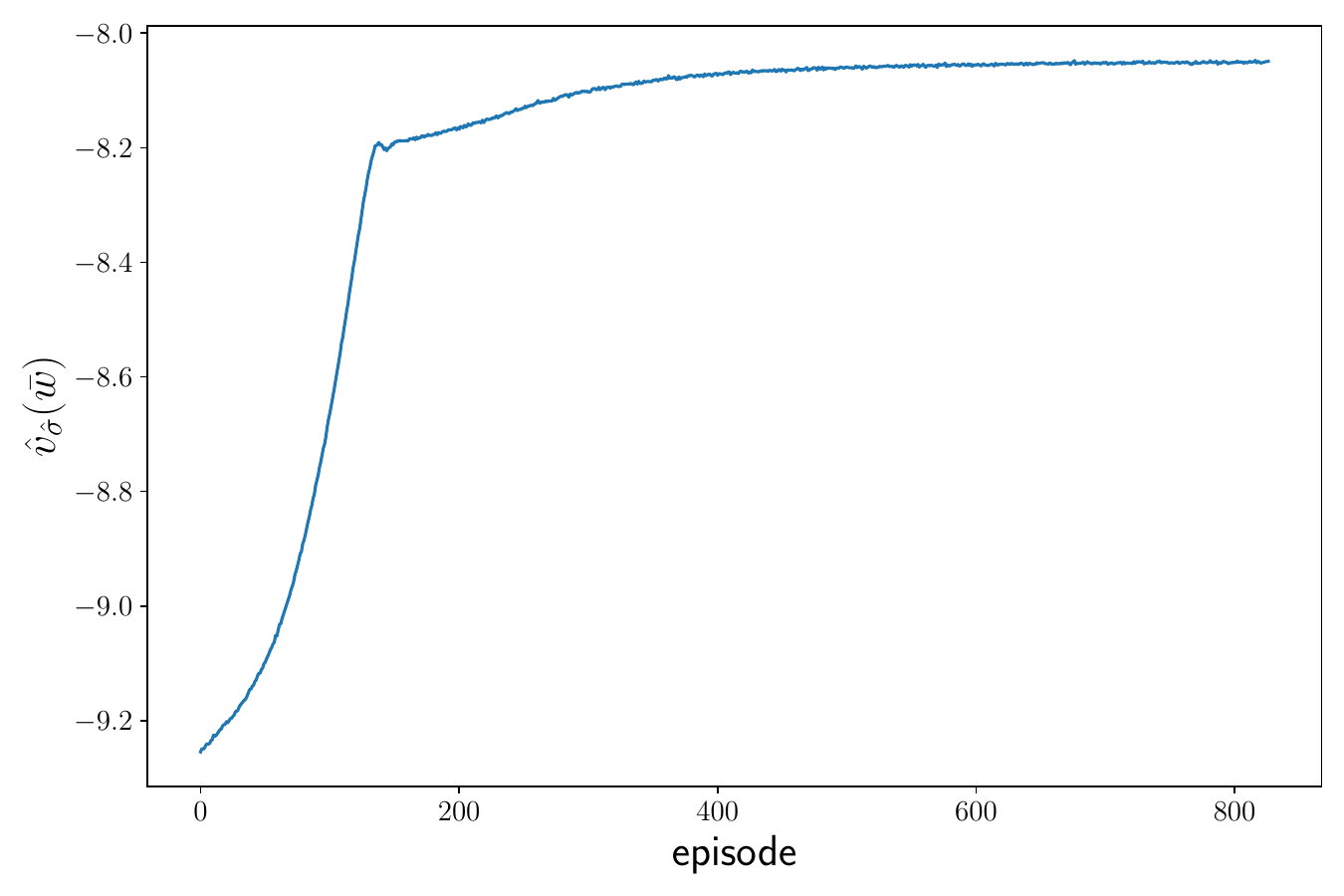}
    \caption{$\hat v_{\hat \sigma}(\bar w)$ over training episode for
    the irreducible optimal saving model at $\bar w = 50$.}
    \label{fig:irr_training_50}
\end{figure}

\section{Reducible Optimal Savings MDP with $\bar w = 50$}\label{app:red_50}

In this appendix, we discuss the result of reducible optimal saving MDPs 
with initial wealth $\bar w = 50$. As shown in Figure~\ref{fig:dpg_vfi_red_50}, 
the policy network $\hat \sigma$ is able to approximate the optimal policy 
$\sigma^*$ in the reachable subset of the state space. This provides strong 
empirical evidence for Lemma~\ref{l:bk_access}.

To emphasize this point, Figure~\ref{fig:wealth_comparison} illustrates one 
sample of the wealth dynamics under the policy network $\hat \sigma$ at 
epoch 750 starting from $\bar w = 1$ and $\bar w = 50$. The wealth process 
initialized at $\bar w = 50$ explores a substantially broader range of wealth 
levels compared to the process starting from $\bar w = 1$.
This wider exploration of the state space directly contributes to the improved 
performance observed in Figure~\ref{fig:dpg_vfi_red_50}.

We note that the two sequences converge to similar patterns after a brief 
initial period because we used identical random seeds when generating shocks 
$\eta_t$ and $y_t$ when training all models. 
This technical choice minimizes the influence of factors other than 
the initial wealth level on the comparisons of model performance. The similarity 
in the wealth trajectories indicates that the policy network gives 
similar consumption decisions at lower wealth levels.

Nevertheless, the policy still exhibits limitations: since consumption must 
be strictly positive (i.e., $\hat \sigma(w) > 0$ for all $w \in \Wsf$), the 
algorithm struggles to accurately approximate the optimal policy for very high 
wealth values which are not reachable under $\bar w = 50$ given the bounded 
shocks. This discrepancy paired with bounded uniform shocks causes the 
$\hat \sigma$-value function to systematically underestimate the true value 
function in the upper tail of the wealth state in the left panel of 
Figure~\ref{fig:dpg_vfi_red_50}.
\begin{figure}
    \centering
    \includegraphics[width=\textwidth]{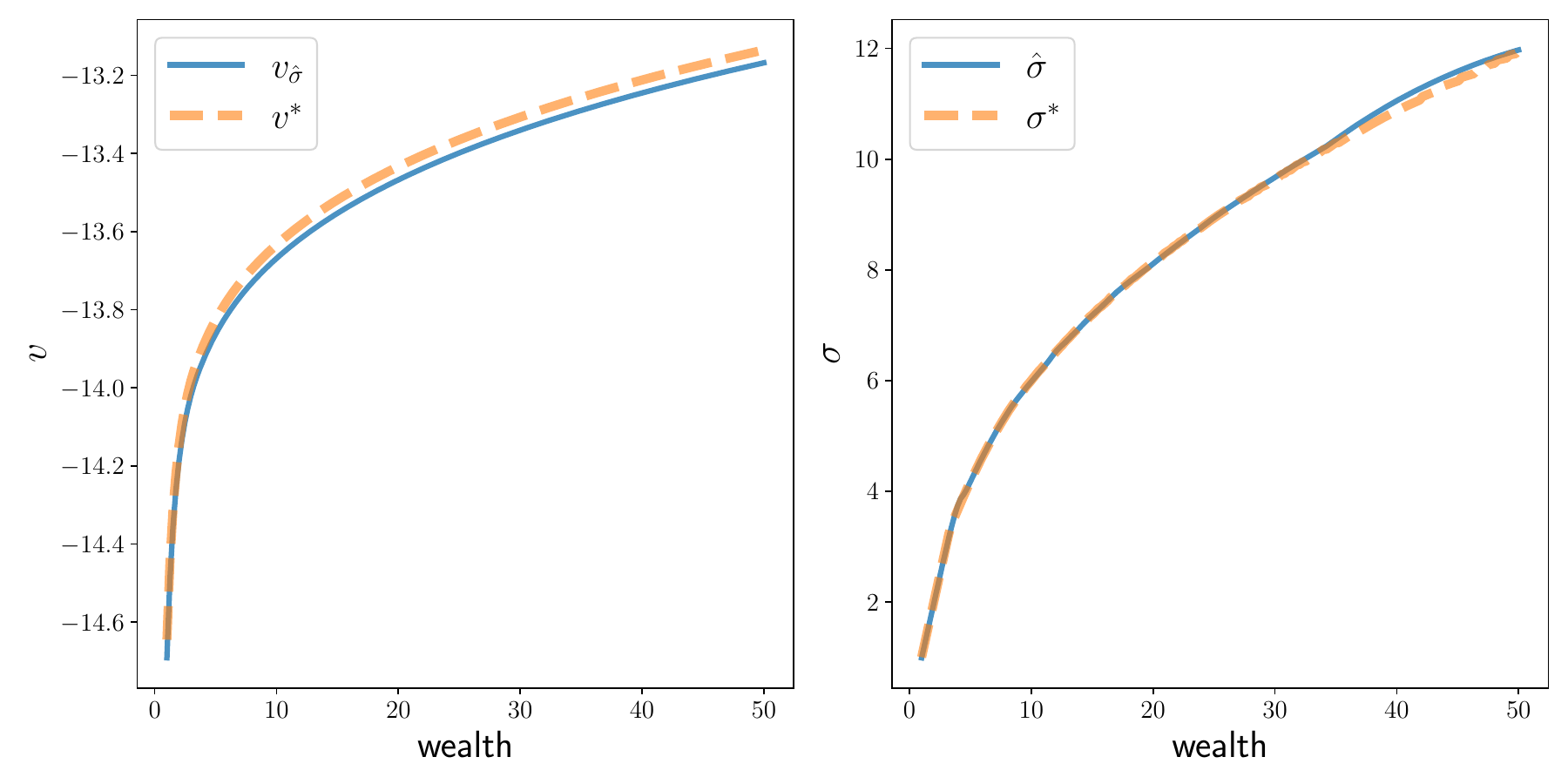}
    \caption{$\hat v_\sigma$ and $\hat \sigma$ with $\bar w = 50$ vs OPI solutions}
    \label{fig:dpg_vfi_red_50}

    \vspace{2em}
    \centering
    \includegraphics[width=0.75\textwidth]{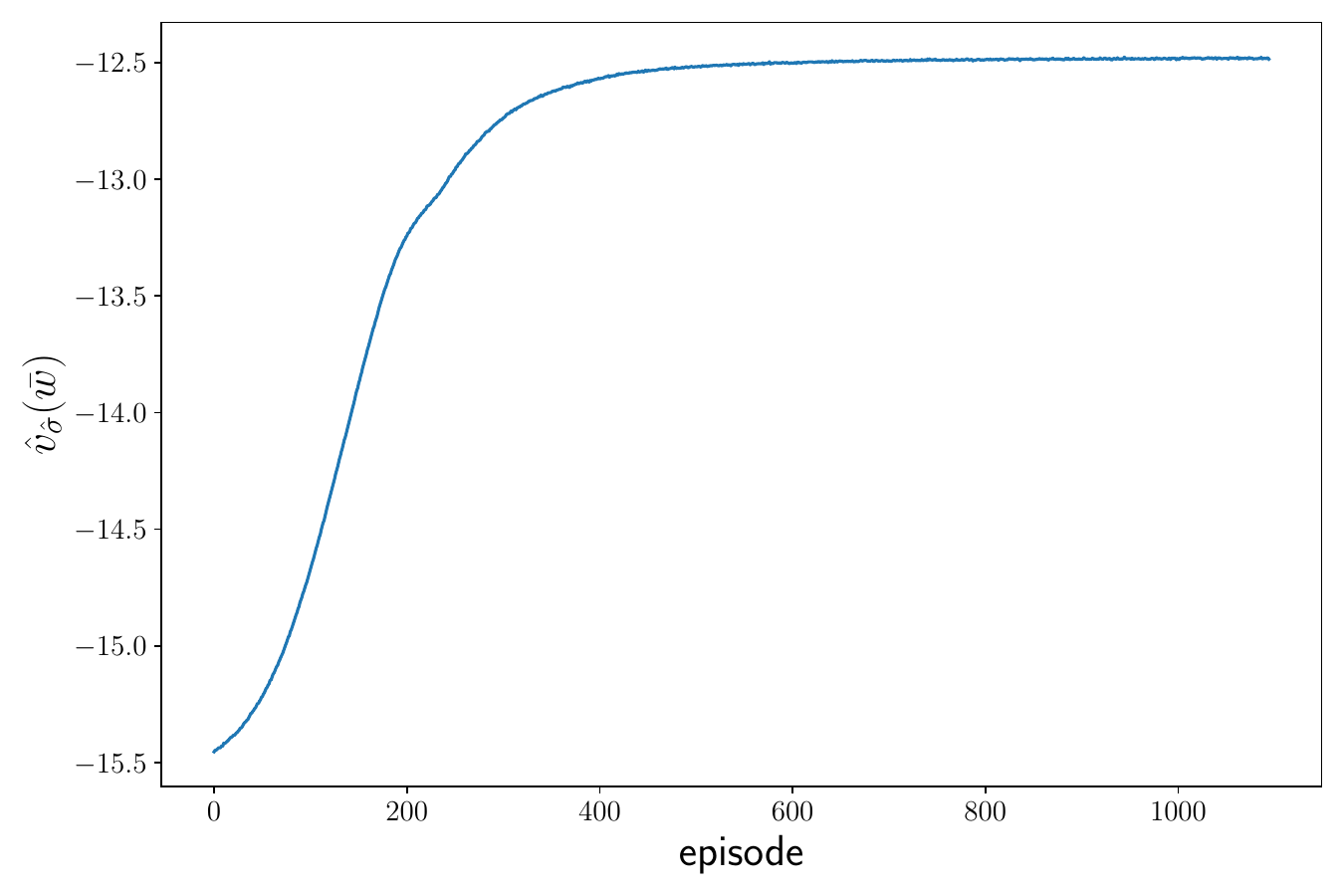}
    \caption{$\hat v_{\hat \sigma}(\bar w)$ over training episode for
    the reducible optimal saving model at $\bar w = 50$.}
    \label{fig:red_training_50}
\end{figure}
These empirical findings emphasize the importance of Lemma~\ref{l:bk_access}, 
which formally establishes that when the irreducibility condition fails, 
the policy network achieves near-optimal performance in 
reachable subsets of the state space starting from $\bar w =50$.

\begin{figure}
    \centering
    \includegraphics[width=0.75\textwidth]{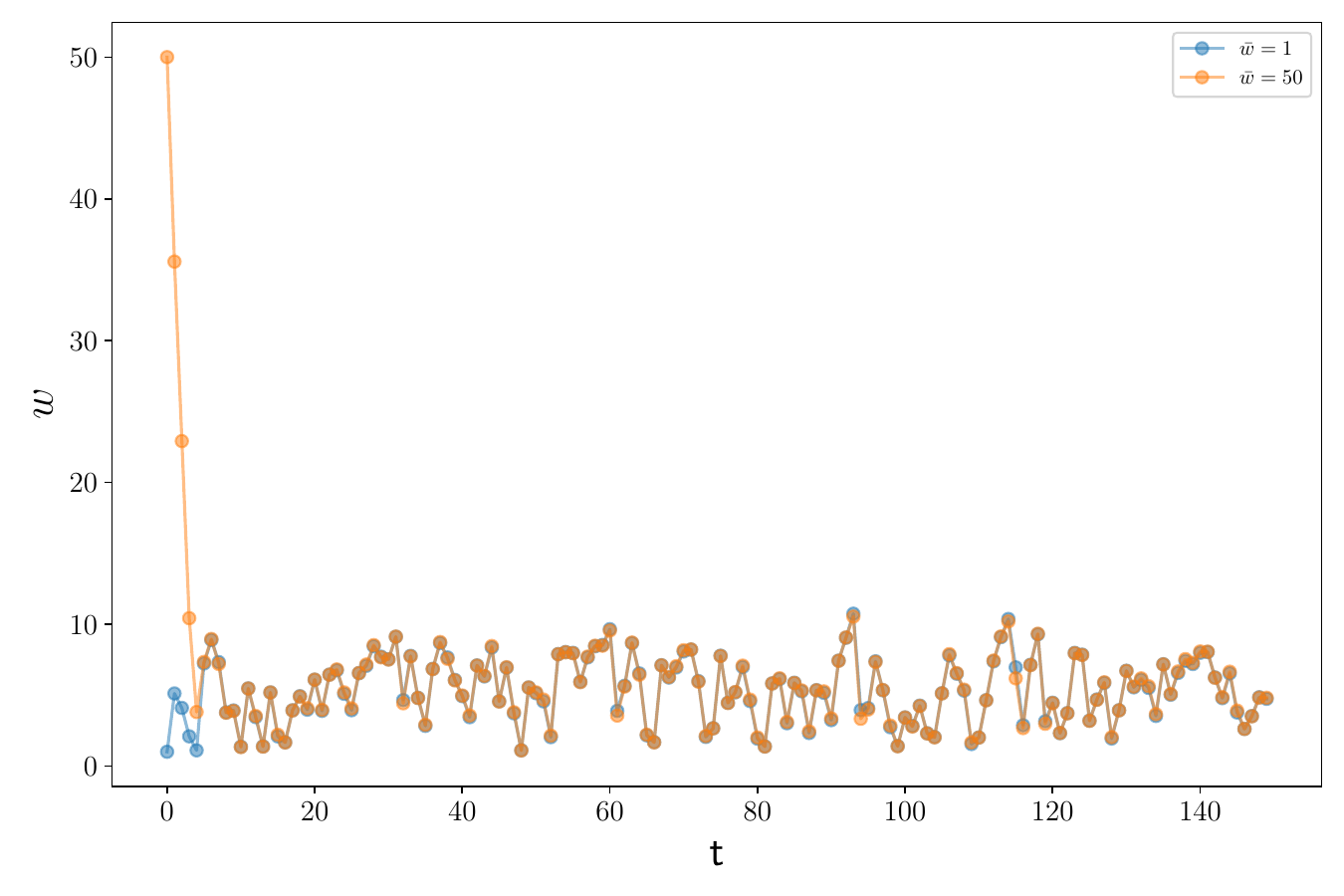}
    \caption{Comparison of wealth trajectories for $\bar w = 1$ and $\bar w = 50$ in 
    a typical rollout.}
    \label{fig:wealth_comparison}
\end{figure}

\end{document}